\documentclass[12pt]{amsart}
\usepackage{amsmath,amssymb,amsthm,wasysym,enumitem}
\usepackage[utf8]{inputenc} 
\numberwithin{equation}{section}

\newtheorem{thm}[subsubsection]{Theorem}
\newtheorem{lem}[subsubsection]{Lemma}
\newtheorem{cor}[subsubsection]{Corollary}
\newtheorem{prop}[subsubsection]{Proposition}

\theoremstyle{definition}
\newtheorem{defn}[subsubsection]{Definition}
\newtheorem{rem}[subsubsection]{Remark}

\def\cL{{\mathcal L}}
\newcommand{\sq}{\mathrel{\square}}

\newcommand{\cF}{\ensuremath{\mathcal{F}}}

\newcommand{\C}{\ensuremath{\mathbb{C}}}

\newcommand{\R}{\ensuremath{\mathbb{R}}}

\newcommand{\N}{\ensuremath{\mathbb{N}}}

\newcommand{\cG}{\ensuremath{\mathcal{G}}}

\newcommand{\Ranp}{\ensuremath{\mathbb{R}_{an}^{pow}}}

\newcommand{\aF}{\ensuremath{\frac{|F^{(\alpha)}_t (0)|}{\alpha !}}}
\newcommand{\h}{\ensuremath{\frac{1}{2}}}

\def\NN{\mathbb{N}}
\def\RR{\mathbb{R}}

\def\ZZ{\mathbb{Z}}

\def\QQ{\mathbb{Q}}

\def\qed{${\vcenter{\vbox{\hrule height .4pt
           \hbox{\vrule width .4pt height 4pt
            \kern 4pt \vrule width .4pt}
             \hrule height .4pt}}}$}
\def\mqed{{\vcenter{\vbox{\hrule height .4pt
           \hbox{\vrule width .4pt height 4pt
            \kern 4pt \vrule width .4pt}
             \hrule height .4pt}}}}

\def\hash{\#}

\title[Parameterizations and diophantine applications]{Uniform parameterization of subanalytic sets and diophantine applications}

\author[Cluckers]{Raf Cluckers}
\address{Universit\'e de Lille,
Laboratoire Painlev\'e,
 CNRS - UMR 8524, Cit\'e Scientifique, 59655
Villeneuve d'Ascq Cedex, France, and,
KU Leuven, Celestijnenlaan 200B, B-3001 Leu\-ven, Bel\-gium}
\email{Raf.Cluckers@univ-lille.fr}
\urladdr{http://rcluckers.perso.math.cnrs.fr/}

\author[Pila]{Jonathan Pila}
\address{Mathematical Institute,
University of Oxford,
Andrew Wiles Building,
Radcliffe Observatory Quarter,
Woodstock Road,
Oxford,
OX2 6GG}
\email{pila@maths.ox.ac.uk}

\author[Wilkie]{Alex Wilkie}
\address{Mathematical Institute,
University of Oxford,
Andrew Wiles Building,
Radcliffe Observatory Quarter,
Woodstock Road,
Oxford,
OX2 6GG}
\email{wilkie@maths.ox.ac.uk}

\begin{document}

\thanks{The authors would like to thank EPSRC for supporting parts of the work under grants EP/J019232/1 and EP/N008359/1, and the Leverhulme Trust for the Emeritus Fellowship EM/2015-002.
The author R.C.~was supported by the European Research Council under the European Community's Seventh Framework Programme (FP7/2007-2013) with ERC Grant
Agreement nr. 615722 MOTMELSUM, and thanks the Labex CEMPI  (ANR-11-LABX-0007-01). The authors would like to thank the referees for numerous valuable suggestions. }

\begin{abstract}
We prove new parameterization theorems for sets definable in the structure $\R_{an}$ (i.e. for globally subanalytic
sets) which are uniform for definable families of such sets. We treat both $C^r$-parameterization and (mild) analytic parameterization.
In the former case we establish a polynomial
(in $r$) bound (depending only on the given family) for the number of parameterizing functions.
However, since uniformity is impossible in the latter case (as was shown by Yomdin via a very simple
family of algebraic sets), we introduce a new notion, {\it analytic quasi-parameterization} (where many-valued
complex analytic functions are used),
which allows us to recover a uniform result.

We then give some diophantine applications motivated by the question as to whether the $H^{o(1)}$ bound in the Pila-Wilkie counting theorem
 can be improved, at least for certain reducts of $\R_{an}$. Both parameterization results are shown to give uniform
$(\log H)^{O(1)}$ bounds
 for the number of rational points of height at most $H$ on  $\R_{an}$-definable Pfaffian surfaces. The quasi-parameterization
 technique produces the sharper result, but the uniform $C^r$-parameterization theorem has the advantage of also applying to $\R_{an}^{pow}$-definable
 families.
\end{abstract}



\dedicatory{\textit{French title.} Paramétrisations uniformes d'ensembles sous-analytiques et applications diophantiennes. }

\maketitle

{\footnotesize{\textit{Résumé.}
Nous démontrons de nouveaux résultats de paramétrisations d'ensembles définissables dans $\R_{an}$ (aussi appelés ensembles sous-analytiques globales), uniformément dans les familles définissables. Nous traitons les paramétrisations $C^r$ ainsi que les paramétrisations douces et analytiques.  Dans le cas $C^r$, nous obtenons une borne polynômiale (en $r$, et dépendant seulement de la famille) pour le nombre de fonctions paramétrisantes. Dans le cas de paramétrisations analytiques, comme l'uniformité est impossible (démontré par Yomdin
pour une famille semi-algébrique très simple), nous introduisons une nouvelle notion de {\it paramétrisations quasi-analytiques} (utilisant les fonctions analytiques complex multi-valuées), ce qui nous permet d'obtenir des résultats uniformes.
Ensuite nous donnons des applications diophantiennes motivées par la question si la borne $H^{o(1)}$ dans le théorème de comptage de Pila-Wilkie  peut être améliorée pour certaines réductions de la structure $\R_{an}$.
Nos deux approches de paramétrisations nous permettent d'obtenir de bornes uniformes de grandeur $(\log H)^{O(1)}$ pour le nombre de points rationels de hauteur au maximum $H$ sur les surfaces Pfaffiennes qui sont $\R_{an}$-définissables. Les paramétrisations quasi-analytiques nous donnent des résultats plus fins, mais les paramétrisations $C^r$ ont l'avantage de fonctionner aussi dans le cadre plus général de familles   $\R_{an}^{pow}$-définissables.}}

\section{Introduction}

The aim of this section is to give an informal account of the results appearing in this paper. Precise definitions
and statements are given in the next section.

So, we are concerned with {\it parameterizations} of bounded definable subsets of real euclidean space. The definability here
is with respect to some fixed (and, for the moment, arbitrary) o-minimal expansion of the real field.
By a parameterization of such a set $X \subseteq \R^n$, we mean a finite collection of definable maps from $(0, 1)^m$ to $\R^n$,
where $m:= \text{dim}(X)$, whose ranges cover $X$. The fact that parameterizations always exist is an easy consequence of the cell
decomposition theorem, but the aim is to construct them with certain differentiability conditions imposed on the parameterizing functions together
with bounds on their derivatives. The first result in this generality was obtained in \cite{PW} (by adapting methods of Yomdin \cite{YY} and Gromov \cite{gromov}
who dealt with the semi-algebraic case), where it was shown that for each positive integer $r$ there exists a parameterization consisting
of $C^r$ functions all of whose derivatives (up to order $r$) are bounded by $1$. Further, the parameterizing functions
may be found {\it uniformly}. This means that if $\mathcal{X} = \{ X_t : t \in T \}$ is a definable family of $m$-dimensional
subsets of $(0, 1)^n$ (say), i.e. the relation ``$t \in T \hspace{2mm} \text{and} \hspace{2mm} x \in X_t$'' is definable in both $x$ and $t$,
then there exists a positive integer $N_r$ such that for each $t \in T$, at most $N_r$ functions are required to parameterize
$X_t$ and each such function is definable in $t$.  (The bound $N_r$ does, of course, also depend on
the family $\mathcal{X}$, but we usually suppress this in the notation. The point is that it is independent of $t$.)
Unfortunately, the methods of \cite{PW} do not give an explicit
bound for $N_r$ and it is the first aim of this paper to do so in the case that the ambient o-minimal  structure
is the restricted analytic field $\R_{an}$  (where the bounded definable sets are precisely the bounded subanalytic sets), or a suitable reduct of it.
We prove, in this case,
that $N_r$ may be taken to be a polynomial in $r$ (which depends only on the given family $\mathcal{X}$). While we have only diophantine
applications in mind here, this result already gives a complete answer to an open question, raised by Yomdin, coming from
the study of entropy and dynamical systems (see e.g. \cite{YY}, \cite{YY2}, \cite{gromov}, \cite{Burguet-Liao-Yang}). In fact, even in the case that the ambient structure is just the
ordered field of real numbers (which is certainly a suitable reduct of $\R_{an}$ to which our result applies), the polynomial bound appears to be new
and, indeed, gives a partial answer to a question raised in \cite{Burguet-Liao-Yang} (just below Remark 3.8); the essential missing ingredient to solve this question completely is an effective form of the preparation result of \cite{MillerD} in the semi-algebraic case.
Our uniform $C^r$-parameterization theorem also holds
for the expansion of $\R_{an}$ by all power functions (i.e. the structure usually denoted $\R_{an}^{pow}$) and suitable reducts (to be clarified
in section 2) of it. In fact, we obtain a pre-parameterization result in Section \ref{sec:a-b-m} which underlies $C^r$-parameterizations.

Next we consider {\it mild} parameterizations. Here it is more convenient to consider parameterizing functions with domain $(-1, 1)^m$
(where $m$ is the dimension of the set being parameterized) and we demand that they are $C^{\infty}$ and we put
a bound on {\it all} the derivatives. We shall only be concerned with functions that satisfy a so called $0$-mild condition, namely
that there exists an $R>1$ such that for each positive integer $d$, all their $d$'th derivatives have a bound of order $R^{-d} \cdot d!$ (which in fact forces the functions
to be real analytic). It was shown in \cite{JMT} that any reduct of $\R_{an}$ has the $0$-mild parameterization property: every definable
subset of $(-1, 1)^n$ has a parameterization by a finite set of $0$-mild functions. However, this result
cannot be made uniform. For Yomdin  showed in \cite[Proposition 3.3]{Y3} (see also \cite[page 416]{Y2}) that the number of $0$-mild functions required to parameterize the set
$\{ \langle x_1 , x_2 \rangle \in (-1, 1)^2 : x_1 \cdot x_2 = t \}$ necessarily tends to infinity as $t \to 0$. Our second parameterization
result recovers uniformity in the $0$-mild setting but at the expense of, firstly, covering larger sets than, but ones having the same dimension as, the sets in the given
family and secondly, covering not by ranges of $0$-mild maps but by solutions to (a definable family of) Weierstrass polynomials with
$0$-mild functions as coefficients.

\vspace{2mm}

In \cite{PW} the parameterization theorem is applied to show that any definable subset of $(0, 1)^n$ (the ambient o-minimal
structure being, once again, arbitrary) either contains an infinite semi-algebraic subset or else, for all $H \geq 1$, contains at most
$H^{o(1)}$  rational points whose coordinates have denominators bounded by $H$. (For the purposes of this introduction we refer
to such points as $H$-{\it bounded} rational points.) Although this result is best possible in general, and is so even for
one dimensional subsets of $(0, 1)^2$ definable in the structure
$\R_{an}$, it has been conjectured that the $H^{o(1)}$ bound may be improved to $(\log H)^{O(1)}$ for certain reducts
of $\R_{an}$ (specifically, for sets definable from restricted {\it Pfaffian} functions), and it is our final aim in this paper to take a small step towards such a conjecture.

We first observe that the point counting theorem from \cite{PW} quoted above follows (by induction on dimension)
from the following uniform result (the main lemma of \cite{PW} on page 610). Namely, if  $m<n$ and $\mathcal{X} = \{ X_t : t \in T \}$ is a definable family of $m$-dimensional
subsets of $(0, 1)^n$, and $\epsilon > 0$,  then there exists a positive integer $d= d(\epsilon, n)$ such
that for each $t \in T$ and for all $H \geq 1$, all the $H$-bounded rational points of $X_t$ are contained in the union of at most $O(H^{\epsilon})$ algebraic
hypersurfaces of degree at most $d$, where the implied constant depends only on $\mathcal{X}$ and $\epsilon$. Now, for the structure $\R_{an}^{pow}$ (or any
of its suitable reducts), our uniform
$C^r$-parameterization theorem allows us to
improve the bound here on the number of hypersurfaces to $O((\log H)^{O(1)})$ (for $H > e$, with the implied constants depending only on the family
$\mathcal{X}$) but, unfortunately, their degrees have this order of magnitude too. Actually, the bound on the degrees is completely explicit, namely
$[(\log H)^{m/(n-m)}]$, but as this tends to infinity with $H$, the inductive argument used in \cite{PW} (where the degree $d$ only depended on $\epsilon$ and $n$)
breaks down at this point.  Our $0$-mild (quasi-) parameterization theorem does give a better result for (suitable reducts of) the structure $\R_{an}$ in that the
number of hypersurfaces is bounded by a constant (depending only on $\mathcal{X}$), but the bound for their degrees is the same as above and so, once again,
the induction breaks down.

We can, however, tease out a uniform result for rational points on certain one and two dimensional sets definable from
restricted Pfaffian functions, but for the general conjecture a completely new uniform parameterization theorem that applies to
the intersection of a definable set of constant complexity
with an algebraic hypersurface of {\it nonconstant} degree is badly needed.

\medskip

{\bf Note.\/} As this paper was being finalised, the arXiv preprints \cite{BN1}, \cite{BN2} appeared.
There is some similarity in the methods used there and the complex analytic approach here.
There seems to be no inclusion in either direction in the parameterization results nor in
 the diophantine applications; the diophantine result of \cite{BN2}  deals with sets of
arbitrary dimension but in a smaller reduct of $\R_{an}$.

\section{Precise statements}

\noindent
2.1 \hspace{1mm} {\bf $C^r$-parameterizations}

The largest expansion of the real field (that is, the expansion with the most definable sets) to which our
uniform $C^r$-parameterization theorem applies is the structure $\R_{an}^{pow}$, i.e. the expansion by all restricted analytic
functions and all power functions $(0, \infty ) \to (0, \infty ): x \mapsto x^s$ (for $s \in \R$).  However, when it comes
to  applications there is considerable advantage to be gained from working in suitable reducts of  $\R_{an}^{pow}$ for
which more effective topological and geometric information is available for the definable sets. (For example, for sets
definable from restricted Pfaffian functions one has, through the work of Khovanskii
(\cite{Khov}) and Gabrielov and Vorobjov (\cite{GV}), good bounds (in terms of natural data) on the number of their connected components.)

It turns out that for our proof here to go through, the property required of the ambient o-minimal structure is that
it should be a reduct of $\R_{an}^{pow}$ in which
a suitable version of the Weierstrass Preparation Theorem holds for definable functions. Now, a large class of such reducts
has been identified and extensively studied by D. J. Miller in his Ph.D. thesis (and in \cite{MillerD}), inspired by the results from \cite{LR} and \cite{Paru2} in the subanalytic case.
These are based on a language for functions in a {\it Weierstrass system} together with a certain class of power functions.
There is no need for us to go into precise definitions here-we will quote the relevant results from \cite{MillerD} when needed. Suffice it to say that examples
include the real ordered field itself, $\R_{an}$, $\R_{an}^{pow}$ or, indeed, the expansion of $\R_{an}$ by any collection of power functions that is closed
under multiplication, inverse and composition (i.e. such that the exponents form a subfield of $\R$). Many more examples appear in the literature (see \cite{Dries:Wei}, \cite{DenLip}
and \cite{MillerD}). We shall assume,  in the precise statement of the theorem below and throughout section 4, that all notions of definability
are with respect to some such fixed reduct of $\R_{an}^{pow}$:

\vspace{2mm}

\noindent
{\bf Convention 2.1.1} We fix a reduct of $\R_{an}^{pow}$ based on a Weierstrass system $\mathcal{F}$ and a
subfield $K$ of its field of exponents as described in \cite[Definition 2.1]{MillerD}. As there, we denote its language by $\mathcal{L}^{K}_{\mathcal{F}}$.

\vspace{2mm}

\noindent
Note that, for the smallest possible choice of $\cF$ and $K=\QQ$, the $\mathcal{L}^{K}_{\mathcal{F}}$-definable sets are precisely the semi-algebraic sets. For the largest possible choice of $\cF$, and for $K=\RR$, one has that $\mathcal{L}^{K}_{\mathcal{F}}$ equals $\R_{an}^{pow}$.

\vspace{2mm}

\noindent
{\bf Definition 2.1.2}  Let $r$ be a non-negative integer or $+\infty$. The $C^r$-norm $\rVert f \lVert_{C^r}$
of a $C^r$-function $f : U \subset \R^m \to \R$, with $U$ open and nonempty, is defined (in $\R \cup \{ + \infty \}$) by
$$\sup_{x \in U} \sup_{\underset{\alpha \in \N^m}{|\alpha| \leq r}} |f^{(\alpha)} (x)|,$$
where $\NN$ is the set of nonnegative integers, and where we have used the standard multi-index notation, namely for $\alpha = \langle \alpha_1 , \ldots , \alpha_m \rangle \in \N^m$,
$f^{(\alpha)}$ stands for $\partial^{\alpha} f / \partial x^{\alpha}$ ($= f$ for $\alpha = 0$), and
$|\alpha|$ denotes $\sum_{i=1}^{m} \alpha_i$. By the $C^r$-norm
of a $C^r$-map $f : U \subset \R^m \to \R^s$, with $U$ open, we mean the maximum of the $C^r$-norms of the component functions of $f$.

\vspace{2mm}

\noindent
{\bf Theorem 2.1.3} (The uniform $C^r$-parameterization theorem.) {\it Let $n, k$ be positive integers and $m$ be a
nonnegative integer with $m \leq n$. Let $\mathcal{X} = \{ X_t : t \in T \}$ be an $\mathcal{L}^{K}_{\mathcal{F}}$-definable
family of $m$-dimensional subsets of $(0, 1)^n$, where $T$ is some $\mathcal{L}^{K}_{\mathcal{F}}$-definable subset
of $\R^k$. Then there exists positive numbers $c$ and $d$, depending only on the family $\mathcal{X}$, such that
for each positive integer $r$,  and for each $t \in T$, there exist analytic maps
$$\phi_{r, i, t} : (0, 1)^m \to X_t$$
for $i = 1, \ldots, cr^d$, whose $C^r$-norms are bounded by $1$ and whose ranges cover $X_t$. Moreover, for each $i$ and $r$,
$\{ \phi_{r,i,t} : t \in T \}$ is an $\mathcal{L}^{K}_{\mathcal{F}}$-definable family of maps. }

\vspace{2mm}

The proof of 2.1.3 is given in section 4 and relies on a pre-parameterization result (Theorem \ref{pre-param}) which underlies $C^r$-parameterizations for all $r$ via power maps.

\vspace{3mm}

\noindent
2.2 \hspace{1mm} {\bf Quasi-parameterization}

For the main result of this section we require our ambient o-minimal structure to be a reduct of $\R_{an}$: we do not know whether theorem 2.2.3 below (or some version
of it) holds if power functions with irrational exponents are admitted. We shall be working with complex valued definable functions
of several complex variables where the definability here is via the usual identification of $\C$ with $\R^2$. Naturally enough we will require
there to be enough {\it definable} holomorphic functions:

\vspace{2mm}

\noindent
{\bf Convention 2.2.1} We fix a reduct of $\R_{an}$ with the following property. If $f : U \subset \R^m \to \R$, with $U$ open, is a definable,
real analytic function, then for each $a \in U$ there exists an open $V \subset \C^m$ with $a \in V \cap \R^m \subset U$ and a {\it definable}
holomorphic function $\tilde{f} : V \to \C$ such that for all $b \in V \cap \R^m$, $\tilde{f}(b) = f(b)$. For the remainder of this subsection and throughout section 5
 (unless otherwise stated) definability will be with respect to this structure.

\vspace{2mm}

The main examples are the real field and $\R_{an}$ itself. Others may be constructed as follows. Let
 $\mathcal{F}$ be a collection of restricted (real) analytic functions closed under partial differentiation and under the operation implicit
in 2.2.1 (i.e. under taking the real and imaginary part functions of the local complex extensions). Then the expansion of the real field
by $\mathcal{F}$ will be a reduct of $\R_{an}$ satisfying 2.2.1. This follows fairly easily from the theorem of Gabrielov (\cite{Gab2}) asserting that
such a reduct is model complete. (For a local description of the complex holomorphic functions that are definable in such a structure (at least,
in a neighbourhood of a generic point) see \cite{Wilkie}.)

\vspace{1mm}

For $R>0$ we denote by $\Delta (R)$ the open disc in $\C$ of radius $R$ and centred at the origin.

\vspace{2mm}

\noindent
{\bf Definition 2.2.2} Let $R>0$, $K>0$ and let $m$ be a positive integer. Then a definable family $\Lambda = \{ F_t : t \in T \}$, where $T$ is a definable
subset of $\R^k$ for some $k$, is called an $(R, m, K)$-family if for each $t \in T$,  the function $F_t : \Delta (R)^m \to \C$ is holomorphic
and for all $z \in \Delta (R)^m$ we have $|F_t(z)| \leq K$.

\vspace{2mm}

We shall develop a considerable amount of theory for such families in section 5. To mention just one result, which is perhaps of independent
interest, we will show that if $R>1$ and, for each $t \in T$,
$$F_t(z) = \sum a_{\alpha}^{(t)} \cdot z^{\alpha}$$
is the Taylor expansion of $F_t$ for $z = \langle z_1 , \ldots, z_m \rangle$ in an open neighbourhood of $0 \in \C^m$ (where the summation is over
all $m$-tuples $\alpha = \langle \alpha_1, \ldots , \alpha_m \rangle
\in \N^m$), then there exists $M = M(\Lambda) \in \N$ such that $|a_{\alpha}^{(t)}|$ achieves its maximum value for some $\alpha$
with $|\alpha| \leq M$. (In addition to the multi-index notation introduced in 2.1.2 we write $z^{\alpha}$ for $z_1^{\alpha_1} \cdots z_m^{\alpha_m}$.)
The fact that $M$ is independent of $t$ here is crucial for all the uniformity results
that follow and leads to the following

\vspace{2mm}

\noindent
{\bf Theorem 2.2.3} (The quasi-parameterization theorem.) {\it Let $n$ and $m$ be non-negative integers with $m<n$ and let $\mathcal{X} =
\{ X_s : s \in S \}$ be a definable family of subsets of $[-1, 1]^n$, each of dimension at most $m$, where $S$ is a definable subset of $\R^k$ for some
$k$. Then there exists $R>1$, $K > 0$, a positive integer $d$ and an $(R, m+1, K)$-family $\Lambda = \{ F_t : t \in T \}$ such that each
$F_t$ is a monic polynomial of degree at most $d$ in its first variable and for all $s \in S$, there exists $t \in T$ such that
$$X_s \subseteq \{ x = \langle x_1 , \ldots , x_n \rangle \in [-1, 1]^n :
\exists w \in [-1, 1]^m \bigwedge_{i=1}^{n} F_t (x_i , w) =0 \}.$$}

The proof of 2.2.3 is given in section 5. However, some of the ideas involved can be illustrated by considering Yomdin's example mentioned in
section 1. Here $n=2$, $m=1$, $S = (0, 1)$ and $X_s = \{ \langle x_1 , x_2 \rangle \in (-1, 1)^2 : x_1 \cdot x_2 = s \}$ (for each $s \in S$).
If we take $T = (0,1)$ and $G_t (z_1 , z_2 ) = z_1^2 - z_2 z_1 + t$ (for $t \in T$), then for all $s \in S$, there exists $t \in T$ such that
$$X_s \subseteq \{ x = \langle x_1 , x_2 \rangle \in [-1, 1]^2 :
\exists w \in [-2, 2] \bigwedge_{i=1}^{2} G_t (x_i , w) =0 \}.$$

(Just take $t=s$ and then, for $\langle x_1,x_2 \rangle \in X_s$, take $w = x_1 + x_2$.)

\vspace{1mm}

The fact that $w$ may not lie in the required interval $[-1,1]$ is an annoying, but entirely superficial, difficulty
that can always be resolved by a process that we will refer to as ``translation and scaling''. In this case no scaling is required:
we just take our $F_t (z_1 , z_2 )$ to be \newline $G_t (z_1 , z_2 ) \cdot G_t (z_1 , z_2 +1) \cdot G_t (z_1 , z_2 -1)$, so that
$\{ F_t : t \in T \}$ is a $(2, 2, 1331)$-family each member of which is a monic polynomial in $z_1$ of degree $6$, and which now
has the required property exactly.

\vspace{2mm}

\noindent
2.3 \hspace{1mm} {\bf Diophantine applications}
The above parameterization results may be applied to
obtain results about the distribution of rational points
on definable sets.

The height of a rational number $q=a/b$ where $a,b\in \ZZ$
with $b>0$ and ${\rm gcd}(a,b)=1$ is defined to
be $H(q)=\max(|a|, b)$ and the height of a tuple
$q=(q_1,\ldots, q_n)\in \QQ^n$ is
$H(q)=\max(H(q_i), i=1,\ldots,n)$. For
$X\subset\RR^n$ we set
$$
X(\QQ, H)=\{x\in X\cap \QQ^n: H(x)\le H\}
$$
and define the {\it counting function\/}
$$
N(X,H)=\hash X(\QQ, H).
$$

It is convenient to express the diophantine applications in the
same settings as the corresponding parameterization results.
Thus 2.3.1 below considers a family
${\mathcal X}\subset T\times (0,1)^n$
of sets $X_t\subset (0,1)^n, t\in T$,
while 2.3.2 considers a family
${\mathcal X}\subset T\times [-1,1]^n$
of sets $X_t\subset [-1,1]^n, t\in T$,
in each case definable in a suitable (specified)
o-minimal structure.
As mentioned above, this is a superficial issue.
We assume that each fibre $X_t$ has dimension $m<n$.

By $[x]$ we denote
the integer part of a real number $x$: $[x]\in\ZZ$
and $[x]\le x < [x]+1$.

\medskip
\noindent
{\bf Theorem 2.3.1}
{\it Let ${\mathcal X}\subset T\times (0,1)^n$ be a family
of sets $X_t, t\in T$, of dimension $m$,
definable in $\RR_{an}^{pow}$. Then there exist positive constants
$C_1=C_1({\mathcal X}), c_1=c_1({\mathcal X})$
such that, for $H\ge e$ and $t\in T$, $X_t(\QQ,H)$
is contained in the union of the zero sets of at most
$$
C_1(\log H)^{c_1}
$$
non-zero polynomials with real coefficients of degree at most\/}
$$
[(\log H)^{m/(n-m)}].
$$

For a definable family in the smaller structure
$\RR_{an}$ we get a more precise result.

\medskip
\noindent
{\bf Theorem 2.3.2 \/}
{\it Let ${\mathcal X}\subset T\times [-1,1]^n$ be a family
of sets $X_t, t\in T$, of dimension $m$, definable in $\RR_{an}$.
Then there exists a positive constant
$C_2=C_2({\mathcal X})$
such that, if $H\ge e$ and $t\in T$ then
$X_t(\QQ, H)$ is contained in the union of the zero sets of at most
$$
C_2
$$
non-zero polynomials with real coefficients of degree at most\/}
$$
[(\log H)^{m/(n-m)}].
$$

\medskip

If these results could be iterated on the intersections
we would be able to prove a bound of the form $(\log H)^{O(1)}$
for rational points up to height $H$,
unless the set contained a positive-dimension a semi-algebraic subset,
as discussed in \S1.
However,  as the degrees of the
hypersurfaces increase with $H$, even a second iteration
would require a result for such non-definable families.

However, for certain families of Pfaffian sets of dimension 2
(see the basic definitions  below) we can carry this out
using estimates due to Gabrielov and Vorbjov \cite{GV}.
They have the right form of dependencies
to give a suitable result for the curves arising when the surface
is intersected with algebraic hypersurfaces of growing degree.
This idea has been used in several previous
papers \cite{Butler}, \cite{JMT}, \cite{JT}, \cite{P2}, \cite{Pila:Mild}.

\medskip
\noindent
{\bf Definition 2.3.3 \/}
A {\it Pfaffian chain\/} of order $r\ge 0$ and degree $\alpha\ge 1$
in an open domain $G\subset \RR^n$ is a sequence of analytic
functions $f_1,\ldots, f_r$ on $G$ satisfying differential equations
$$
df_j=\sum_{i=1}^ng_{ij}(x, f_1(x),\ldots, f_j(x))dx_i
$$
for $1\le j\le r$, where $g_{ij}\in \RR[x_1,\ldots, x_n, y_1,\ldots, y_j]$
are polynomials of degree not exceeding $\alpha$. A function
$$
f=P(x_1,\ldots, x_n, f_1,\ldots, f_r)
$$
where $P$ is a polynomial in $n+r$ variables with real coefficients
of degree not exceeding $\beta\ge 1$
is called a {\it Pfaffian function\/} of order $r$ and degree $(\alpha, \beta)$.
A {\it Pfaffian set\/} will mean the set of common zeros
of some finite set of Pfaffian functions.

\medskip

By $\RR_{\rm Pfaff}$ we mean the expansion of the real
ordered field $\RR$
by all Pfaffian functions $f:\RR^n\rightarrow \RR, n=1,2,\ldots$.
This is an o-minimal structure \cite{WilkiePfaff}.
The smaller o-minimal structure $\RR_{\rm resPfaff}$
is the expansion of $\RR$ by all functions of the form $f\vert_{[0,1]^n}$ where
$f: G\rightarrow\RR$ is a Pfaffian function and $[0,1]^n\subset G$.

The following notion of a ``Pfaffian surface'' is much more restrictive then
a two-dimensional set definable in $\RR_{\rm Pfaff}$.

\medskip

\noindent
{\bf Definition 2.3.4 \/}
By a {\it Pfaffian surface\/} we will mean the union of the
graphs in
$\RR^3$ of finitely many Pfaffian functions of two variables
with a common Pfaffian chain of order and degree $(r, \alpha)$,
defined on a ``simple'' domain $G$ in the sense of \cite{GV}.
Namely, a domain of the form $\RR^2, (-1,1)^2, (0, \infty)^2$
or $\{(u,v): u^2+v^2<1\}$.
We take the {\it complexity\/}
of the surface to be the triple $(r, \alpha, \beta)$,
where $\beta$ is the maximum of the degrees of the
Pfaffian functions defining the surface.

\medskip
\noindent
{\bf Definition 2.3.5 \/} Let $X\subset \RR^n$.
The {\it algebraic part\/}
of $X$, denoted $X^{\rm alg}$ is the union of all connected
positive dimensional semi-algebraic subsets of $X$.
The complement $X-X^{\rm alg}$ is called the {\it transcendental part\/}
of $X$ and denoted $X^{\rm trans}$.

\medskip

By combining 2.1.3 with the methods of \cite{P2}, \cite{Pila:Mild} we get
a uniform result  for a family of Pfaffian surfaces definable in
$\RR_{an}^{pow}$.
For an individual surface definable in the structure
$\RR_{\rm resPfaff}$ such a bound is due
to Jones-Thomas \cite{JT}.
Perhaps a combination of the methods could give
uniformity for $\RR_{an}^{pow}$-definable families of
restricted-Pfaffian-definable sets
of dimension 2.

\medskip
\noindent
{\bf Proposition 2.3.6 \/} {\it
Let $r$ be a non-negative integer and $\alpha, \beta$
positive integers.
Let ${\mathcal X}\subset T\times (0,1)^3$ be a family
of surfaces $X_t, t\in T$, definable in
$\RR_{an}^{pow}$ such that each fibre $X_t$
is the intersection of $(0,1)^3$ with
a Pfaffian surface of complexity  (at most) $(r, \alpha, \beta)$.
Then there exist $C_3({\mathcal X}), c_3({\mathcal X})$
such that, for $H\ge e$ and $t\in T$,
$$
N(X_t^{\rm trans}, H)\le C_3(\log H)^{c_3}.
$$
}

\medskip

When the family is definable in $\RR_{an}$  we can prove a more
precise uniform result in which the exponent depends
only on the complexity of the Pfaffian surfaces.

\medskip
\noindent
{\bf Proposition 2.3.7 \/} {\it
Let $r$ be a non-negative integer and $\alpha, \beta$
positive integers.
Let ${\mathcal X}\subset T\times [-1,1]^3$ be a family
of surfaces $X_t, t\in T$, definable in
$\RR_{an}$ such that each fibre $X_t$
is the intersection of $[-1,1]^3$ with
a Pfaffian surface of complexity  (at most)
$(r, \alpha, \beta)$.
Then there exist $C_4({\mathcal X})$ and
$c_4(r, \alpha, \beta))$
such that, for $H\ge e$ and $t\in T$,
$$
N(X_t^{\rm trans}, H)\le C_4(\log H)^{c_4}.
$$
}

\medskip

The proofs of Theorems 2.3.1 and 2.3.2 and Propositions 2.3.6 and 2.3.7,
assuming the parameterization results,  are given in Section 3.

\section{Proofs of diophantine applications}

\subsection{Some preliminaries for 2.3.1 and 2.3.2}

For a positive integer $k$ and non-negative integer
$\delta$ we let
$$
\Lambda_k(\delta)=\{\mu=(\mu_1,\ldots,\mu_k)\in\NN^k:
|\mu|=\mu_1+\ldots+\mu_k=\delta\},
$$
$$
\Delta_k(\delta)=\{\mu=(\mu_1,\ldots,\mu_k)\in\NN^k:
|\mu|=\mu_1+\ldots+\mu_k\le\delta\},
$$
$$
L_k(\delta)=\hash\Lambda_k(\delta),\quad
D_k(\delta)=\hash\Delta_k(\delta).
$$
We recall that $\NN$ denotes the set of nonnegative integers.

Let $X=X_t$ be a fibre of our definable family.
We will adapt the methods of \cite{Pila:Mild}, in which
we explore $X(\QQ, H)$ with hypersurfaces of degree
$$
d=[(\log H)^{m/(n-m)}].
$$
This leads us to consider $D_n(d)\times D_n(d)$
determinants $\Delta$ whose entries are the monomials
of degree $d$ (indexed by $\Delta_n(d)$) evaluated
at $D_n(d)$ points of $X$. These points lie on
some algebraic hypersurface of degree $d$ if and
only if $\Delta=0$.

Given some suitable parameterization of $X$ by functions of
$m$ variables, we estimate the above determinant
by a Taylor expansion of the monomial functions
to a suitable order $b$ (remainder term order $b+1$).
The order of the Taylor expansion will match the size
of the matrix, and so we define $b(m,n,d)$ as the unique
integer $b$ with
$$
D_k(b)\le D_n(d) < D_k(b+1).
$$
It is an elementary computation, carried out in
\cite{Pila:Mild}, that
$$
b=b(m,n,d)=\left(\frac{m!}{n!}\right)^{1/m}d^{n/m}(1+o(1))
$$
where the $o(1)$ means, here and below, as
$d\rightarrow\infty$ with $m,n$ fixed. In particular,
$$
b(m,n,d)+1\le 2\left(\frac{m! d^n}{n!}\right)^{1/m}
$$
provided $d\ge d_0(m,n)$ and hence provided
$H\ge H_0(m,n)$.

The fact that $b$ is rather larger than $d$ is crucial
to the estimates.

\subsection{Proof of 2.3.1}

In this and subsequent subsections,
$C,c,\ldots$ will denote constants depending on
${\mathcal X}$, while $E$ denotes a constant depending
only on $m,n$, and in both cases they may differ
at each occurrence.

Let $X=X_t$ be a fibre of ${\mathcal X}$.
We assume for now that $H\ge H_0(m,n)$
for some $H_0(m,n)$ to be specified in the
course of the proof.

According to Theorem 2.1.3, we can parameterize
$X$ by functions
$$
\phi: (0,1)^m\rightarrow(0,1)^n
$$
such that
all partial derivatives of all component functions up to
degree $b+1$ are bounded in absolute value by 1,
and we can cover $X$ using at most
$$
Cb^{c}
$$
such functions, where $C, c$ depend on ${\mathcal X}$.

Let us fix one such function $\phi=(\phi_1,\ldots, \phi_n)$,
where $\phi_i: (0,1)^m\rightarrow (0,1)$.
From now on we deal only with $\phi$.
Our bounds will depend only on bounds for the derivatives
of the coordinate functions of $\phi$ up to order $b+1$.
Since $b$ depends on $n,m,d$, from now on,
all constants will depend only on $m, n, d$.

We consider a $D_n(d)\times D_n(d)$ determinant of the form
$$
\Delta=\det\left((x^{(\nu)})^\mu\right)
$$
with $\nu=1,\ldots, D_n(d)$ indexing rows
and $\mu\in\Delta_d(n)$ indexing columns, where
$x^{(\nu)}=(x^{(\nu)}_1,\ldots,x^{(\nu)}_n)\in X(\QQ, H)$
are points in the image of $\phi$,
say $x^{(\nu)}=\phi(z^{(\nu)})$
where $z^{(\nu)}\in (0,1)^m$,
later to be taken to be in a small disc in $(0,1)^m$,
and $x^\mu=\prod x_i^{\mu_i}$.

As each $x_j^{(\nu)}$ is a rational number with denominator
$\le H$, we find that there is a positive integer
$K$ such that $K\Delta\in \ZZ$ and
$$
K\le H^{ndD_n(d)}. \eqno{(*)}
$$

If we write
$$
\Phi_\mu =\prod_{i=1}^n \phi_i^{\mu_i}
$$
for the corresponding monomial function on the $\phi_i$
then we have
$$
\Delta=\det\left(\Phi_\mu(z^{(\nu)})\right).
$$
We now assume that the $z^{(\nu)}$ all lie in a small
disc of radius $r$ centred at some $z^{(0)}$
and expand the $\Phi_\mu$ in Taylor polynomials to order
$b$ (with remainder term of order $b+1$).
For $\alpha\in\Delta_k(b), \beta\in\Lambda_k(b+1)$ we write
$$
Q_{\nu, \mu}^\alpha=
\frac{\partial^\alpha\Phi_\mu(z^{(0)})}{\alpha!}
\left(z^{(\nu)}-z^{(0)}\right)^\alpha,\quad
Q_{\nu, \mu}^\beta=
\frac{\partial^\beta\Phi_\mu(z^{(0)})}{\beta!}
\left(\zeta^{(\nu)}_\mu-z^{(0)}\right)^\beta
$$
with a suitable intermediate point $\zeta^{(\nu)}_\mu$
(on the line joining $z^{(0)}$ to $z^{(\nu)}$),
and with $\alpha!=\prod_{j=1}^k\alpha_j!$,
so that the Taylor polynomial is
$$
\Phi_\mu(z^{(\nu)})=\sum_{\alpha\in \Delta_k(b)}
Q_{\nu, \mu}^\alpha
+\sum_{\beta\in\Lambda_m(b+1)}
Q_{\nu, \mu}^\beta
$$
and we have
$$
\Delta=\det\left(\sum_{\alpha\in\Delta_m(b+1)}
Q_{\nu, \mu}^\alpha\right).
$$

We expand the determinant as in \cite{P1}
(see also \cite{Scanlon1}, eqn (2), p48), using column linearity to get
$$
\Delta=\sum_{\tau}
\Delta_\tau,\quad
\Delta_\tau=\det\left(Q_{\nu, \mu}^{\tau(\mu)}\right).
$$
with the summation over
$\{{\tau: \Delta_n(d)\rightarrow
\Delta_m(b+1)}\}$.

Now if, for some $k\in\{1,\ldots, b\}$, we have
$$
\hash \tau^{-1}(\Lambda_m(k))> L_m(k) \eqno{(**)}
$$
then $\Delta_\tau=0$ as the corresponding
columns are dependent (the space of homogeneous forms
in $(z^{(\nu)}-z^{(0)})$ of degree $k$ has rank
$L_m(k)$). Thus all
surviving terms have a high number of factors
of the form $(z^{(\nu)}-z^{(0)})$ and/or
$(\zeta^{(\nu)}_\mu-z^{(0)})$. We quantify this.

The function
$\Phi_\mu$ is a product of $|\mu|\le d$ functions
$\phi_i$, which have suitably bounded derivatives.
Let us consider more generally a function
$$
\Theta=\prod_{i=1}^\delta\theta_i
$$
where $\theta_i$ have $|\theta_i^{(\alpha)}(z)|\le 1$
for all $|\alpha|\le b+1$. Then, for $\alpha$ with
$|\alpha|\le b+1$, we have
$$
\Theta^{(\alpha)}=
\sum_{\alpha^{(i)}}
{\rm Ch}(\alpha^{(1)},\ldots,\alpha^{(\delta)})
\prod_{i=1}^\delta \theta^{(\alpha^{(i)})}
$$
with the summation
over $\alpha^{(1)}+\ldots+\alpha^{(\delta)}=\alpha$
where
$$
{\rm Ch}(\alpha^{(1)},\ldots,\alpha^{(\delta)})=
\prod _{j=1}^m\left(\frac{\alpha_j!}{
\alpha^{(1)}_j!\ldots\alpha^{(\delta)}_j!}\right).
$$
Since $|\theta_i^{(\alpha)}(z)|\le 1$
for all $|\alpha|\le b+1$ we have
$$
|\Theta^{(\alpha)}(z)|\le
\sum_{\alpha^{(i)}_1}
\left(\frac{\alpha_1!}{
\alpha^{(1)}_1!\ldots\alpha^{(\delta)}_1!}\right)\times
\ldots\times\sum_{\alpha^{(i)}_m}
\left(\frac{\alpha_m!}{
\alpha^{(1)}_m!\ldots\alpha^{(\delta)}_m!}\right)
$$
each summation subject to $\sum_{i} \alpha^{(i)}_j=\alpha_j$
hence
$$
|\Theta^{(\alpha)}(z)|\le
\delta^{\alpha_1}\cdots\delta^{\alpha_m}=\delta^{|\alpha|}.
$$
Therefore
$$
|Q_{\nu, \mu}^\alpha|\le
\frac{(|\mu|r)^{|\alpha|}}{\alpha!}
\le e^{m|\mu|}r^{|\alpha|} \le e^{md}r^{|\alpha|},
$$
and for a $\tau$ which avoids the condition
$(**)$ above
(under which $\Delta_\tau=0$) all terms in its
expansion are bounded above in size by
$$
e^{mdD_n(d)}r^B
$$
where
$$
B=B(m,n,d)=\sum_{k=0}^bL_m(k)k+
\left(D_n(d)-\sum_{\kappa=0}^bL_m(\kappa)\right)(b+1).
$$
Note that
$\left(D_n(d)-\sum_{\kappa=0}^bL_m(\kappa)\right)\ge 0$
by our choice of $b$.
We have the asymptotic expression (see \cite{Pila:Mild})
$$
B=B(m,n,d)=\frac{1}{(m+1)!(m-1)!}\left(
\frac{m!}{n!}\right)^{(m+1)/m}d^{n+n/m} (1+o(1)).
$$

The number of terms from all the $\Delta_\tau$ is
$D_m(b+1)^{D_n(d)}D_n(d)!$
and we conclude that
$$
|\Delta|\le D_m(b+1)^{D_n(d)}D_n(d)!e^{mdD_n(d)}r^B.
$$
Thus we have an integer $K\Delta$ with
$$
K|\Delta|\le
H^{ndD_n(d)}D_m(b+1)^{D_n(d)}D_n(d)!e^{mdD_n(d)}r^B
$$
and if $K|\Delta|\ge 1$ then so is its $B$th root.
Now with our choice of $d$ we find that
$$
\frac{ndD_n(d)}{B}=E\frac{d^{n+1}}{d^{n+n/m}}(1+o(1))=
Ed^{-(n-m)/m}(1+o(1))
$$
(where $E$, according to the convention described above, is a constant depending
only on $n,m$, and possibly different in each occurence) thus
$$
H^{ndD_n(d)/B}\le E
$$
since $d\ge 0.5(\log H)^{(m/(n-m)}$, say, for $H$ sufficiently
large in terms of $n,m$.

The remaining terms
$$
\left(D_m(b+1)^{D_n(d)}D_n(d)!e^{mdD_n(d)}\right)^{1/B}
$$
are also bounded as $d\rightarrow\infty$
(see \cite{Pila:Mild}) and so
$$
\left(K|\Delta|\right)^{1/B}\le Er
$$
where $E$ is a constant depending only on
$n,m$, provided $H\ge H_0(n,m)$ is sufficiently
large. If $rE<1$ then all
points of $X(\QQ, H)$ parameterized by $\phi$ from this disk
lie on one algebraic hypersurface of degree $d$, because
the rank of the rectangular matrix formed by
evaluating all monomials of degree $\le d$ at all such
points is less than $D_n(d)$.

Since $(0,1)^m$ may be covered by some $E'$ such discs,
and there are $C(b+1)^c\le C'(\log H)^{c'}$ maps
$\phi$ which cover $X$, the required
conclusion follows for $H\ge H_0(m,n)$. However for
$H\le H_0$ the number of points is bounded
depending only on $H_0, m, n$.\ \qed

\medskip

\noindent
{\bf Remark.\/} Note that the statements in
\cite[3.2 and 3.3]{Pila:Mild} tacitly assume that
the mildness parameter $A$ satisfies $A\ge 1$,
and that with no conditions on $A$ they can fail
for small height.

\subsection{Setup for 2.3.2}

This and the subsequent two subsections are devoted to
the proof of Theorem 2.3.2.
After some preliminary results, the
proof itself is in 3.5.

We have a family of sets $X_t\subset [-1,1]^n$,
of dimension $m$, definable in $\RR_{an}$. By 2.2.3 we may assume
this family is contained in a family given by a quasi-parameterization,
which we may take to be of the following form.
There exist a positive integer $N$ (which is the degree $d$ of the polynomial
dependence of the functions $F_t$ in their first variable in 2.2.3)
and analytic functions
$$
h_{i,j}: [-1,1]^{m+\eta}\rightarrow \RR, \quad
i=1,\ldots, n,\quad  j=0,\ldots, N-1,
$$
converging on a disc of radius $r_0>1$ in the first $m$ variables.
The remaining variables for $[-1, 1]^\eta$ are for the parameters
of the quasi-parameterization. We have functions
$$
u_j: T\rightarrow [-1,1]^\eta
$$
(which need not be definable)
such that, setting $u(t)=(u_1(t),\ldots, u_\eta(t))$,
for all $t\in T$ and $x=(x_1,\ldots, x_n)\in X_t$
there exists $w=(w_1,\ldots, w_m)\in [-1,1]^m$ such that
$$
x_i^N=h_{i,N-1}(w_i, u(t))x_i^{N-1}+ \ldots +
h_{i, 0}(w_i, u(t)).
$$
The functions $h_{ij}$ record the polynomial dependence of the functions
in 2.2.3 on their first variable; according to 2.2.3 we can
in fact assume that the $h_{i,j}$ are independent of $i$,
but we don't need this.

We keep the previous convention regarding constants.

\subsection{Preliminary estimates}

For each $i$, setting $x=x_i$ and $h_j=h_{i,j}$ and
suppressing the subscript $i$ and the arguments
of the $h_j$, we have a relation
$$
x^N=h_{N-1}x^{N-1}+ \ldots + h_{0}.
$$
By means of this relation, all powers $x^\nu, \nu\in\NN$
may be expressed
as suitable linear combinations of $1,x,\ldots, x^{N-1}$, namely
$$
x^\nu= \sum_{j=0}^{N-1} q_{\nu, j}x^j
$$
with coefficients $q_{\nu, j}\in \ZZ[h_0,\ldots, h_{N-1}]$.
In particular $q_{N,j}=h_j$. We need an estimate
for the degree and integer coefficients of the $q_{\nu, j}$.

Let $H$ be the $N\times N$ matrix of analytic functions

$$
\left(
\begin{array}{cccccc}
0&0&0&\ldots&0&h_0\\
1&0&0&\ldots&0&h_1\\
0&1&0&\ldots&0&h_2\\
\vdots&\vdots&\vdots&\ddots&\vdots&\vdots\\
0&0&0&\ldots&1&h_{N-1}
\end{array}
\right)
$$

Then $H$ acts as a linear transformation on the vector space
with (ordered) basis $\{1, x, \ldots, x^{N-1}\}$ and the $q_{\nu, j}$
are the entries of the column vector
$$
H^\nu
\left(
\begin{array}{c}1\\0\\\vdots\\0
\end{array}
\right)
$$
Inductively, the entry $H^\nu_{\ell j}$ is in
$\ZZ[h_0,\ldots, h_{N-1}]$ of
degree $\max(0, j-N+\nu)$ and the sum of at most
$\max(2^{\nu-1+j-N},1)$ pure (i.e. with coefficient 1)
monomials. We have proved the following.

\medskip
\noindent
{\bf Lemma 3.4.1 \/} {\it For all $\nu\in\NN$ and $j=0,\ldots, N-1$,
$q_{\nu, j}$ is a sum of at most $2^\nu$ monomials of
degree at most $\nu$ in the $h_j$.\ \qed\/}

\medskip

The above is for one variable. We now return to the multivariate
setting with $x=(x_1,\ldots, x_n)$.
For $\nu\in \NN$ we then have
$$
x_i^\nu=\sum_{j=0}^{N-1}q_{i, \nu, j} x^j
$$
where $q_{i, \nu, j}$ is the previously labelled $q_{\nu, j}$ for the
relevant $h_j=h_{i,j}$.

If $\nu\in\NN^n$ we have
$$
x^\nu=\prod_{i=1}^nx_i^{\nu_i}=
\prod_{i=1}^n\sum_{j=0}^{N-1}q_{i, \nu_i ,j}x^j=
\sum_{\lambda\in M}q_{\nu,\lambda}x^\lambda
$$
where $q_{\nu,\lambda}=\prod_{i=1}^nq_{i, \nu_i, \lambda_i}$.

We now want to bound the derivatives (in the $w$ variables)
of the $q_{\mu, \lambda}$. For
$\alpha=(\alpha_1,\ldots, \alpha_m)\in \NN^m$ we set
$\overline{\alpha}=\max(\alpha_i)$.

\medskip
\noindent
{\bf Lemma 3.4.2 \/} {\it
For suitable constants $C, R$ and all $\alpha\in \NN^m$ we have
$$
\frac{|q_{\nu, \lambda}^{(\alpha)}(w, u)|}{\alpha ! }
\le 2^{|\nu|} (\overline{\alpha}+1)^{(|\nu|-1)m}
C^{|\nu|}  R^{|\alpha|}.
$$

}
\medskip
\noindent
{\bf Proof.\/}
For derivatives (in the $w$-variables) of the $h_{i,j}$ we
have a bound of the form
$$
\frac{|h_{i,j}^{(\alpha)}(w, u)|}{\alpha!} \le  \,
C\,R^{|\alpha|},
$$
where $\alpha\in \NN^m$, valid for every $i,j$,
by Cauchy's theorem, since they are
analytic on some disc of radius $r_0>1$.

We have that $q_{i,\nu_i, j}$ is a sum of at most $2^{\nu_i}$
monomials in the $h_{i,j}$, each of degree at most $\nu_i$.
Therefore $q_{\nu, \lambda}$ is a sum of at most
$2^{|\nu|}$ monomials each of degree $|\nu|$ in the $h_{i,j}$.
Consider one such monomial
$$
g=\prod_{h=1}^\ell \phi_h
$$
where each $\phi_h\in\{h_{i,j},: i=1,\ldots, n, j=0,\ldots, N-1\}$
and $\ell\le |\nu|$. As before
$$
\partial^\alpha g =\partial^\alpha \phi_1\ldots\phi_\ell=
\sum_{\alpha^{(1)}+\ldots + \alpha^{(\ell)}=\alpha}
{\rm Ch\/}(\alpha^{(1)},\ldots,\alpha^{(\ell)})
\prod_{h=1}^\ell \partial^{\alpha^{(h)}}\phi_h.
$$
Thus
$$
\frac{\partial^\alpha g}{\alpha!}\le
\sum_{\alpha^{(1)}+\ldots + \alpha^{(\ell)}=\alpha}
\prod_{h=1}^\ell \frac{\partial^{\alpha^{(h)}}\phi_i} {\alpha^{(h)}!}
\le C^{|\nu|} R^{|\alpha|}
\sum_{\alpha^{(1)}+\ldots + \alpha^{(\ell)}=\alpha}1.
$$
The number of summands is at most
$(\overline{\alpha}+1)^{(|\nu|-1)m}$.\ \qed

\medskip
\noindent
{\bf Corollary 3.4.3 \/} {\it For $\alpha\in \NN^m$ we have
$$
\frac{|q_{\nu, \lambda}^{(\alpha)}(w, u)|}{\alpha! }
\le 2^{|\nu|} (|\alpha|+1)^{|\nu|m}
C^{|\nu|}  R^{|\alpha|}.\ \mqed
$$
}

\subsection{Proof of Theorem 2.3.2}

We let $X=X_t$ be a fibre of ${\mathcal X}$.
We will explore $X(\QQ, H)$ with
real algebraic hypersurfaces of degree
$$
d=[(\log H)^{m/(n-m)}],
$$
and again consider $D_n(d)\times D_n(d)$ determinats
$$
\Delta=\det \left((x^{(\nu)})^\mu\right)=0
$$
where $\mu\in \Delta_n(d)$ indexes monomials and
$x^{(\nu)}, \nu=1,\ldots, D_n(d)$ are points of $X$.

For each $x^{(\nu)}$ there is some $w^{(\nu)}$ such that the
quasi-parmeterization conditions hold, i.e.
``$x^{(\nu)}$ is parameterized by the point $w^{(\nu)}$''.
Later we will assume that all the $w^{(\nu)}$ are in
the disc of radius $r$ entered at some $w^{(0)}$.

By our assumptions (see $(*)$ above), there is a positive integer
$K\le H^{ndD_n(d)}$
such that
$$
K\Delta\in \ZZ.
$$

Now let $M=\{\lambda\in\NN^n: \lambda_i<N, i=1,\ldots, n\}$.
Then if $x\in X$ we have
$$
x^\mu =\sum_{\lambda\in M} x^\lambda q_{\mu,\lambda}(w, u(t))
$$
for some $w\in [-1,1]^m$ where
$$
q_{\mu, \lambda}=\prod_{i=1}^n q_{\mu_i, \lambda_i}.
$$

There is a unique $b=b(m,n,d, N)$ such that
$$
N^nD_m(b)\le D_n(d)\le N^nD_m(b+1).
$$
Since $m<n$ there are fewer monomials in $m$ variables
than in $n$ variables, and so if $d$ suitably large in terms of $N, n$
then $b$ is  somewhat larger than $d$.
Set
$$
B(m,n,d, N)=\sum_{\beta=0}^b N^nL_m(\beta)\beta +
\Big(D_n(d)-{N^n}\,\sum_{\beta=0}^b L_m(\beta)\Big)(b+1).
$$
Since $b$ is somewhat larger than $d$, we will have
that $B$ is
somewhat larger than $ndD_n(d)$ (as $d\rightarrow\infty$),
as will be crucial.

Now we assume that all the $w^{(\nu)}$ are all in
the disc of radius $r$ entered at some $w^{(0)}$.
We expand each $q=q_{\mu, \lambda}$ in a Taylor polynomial
with remainder term of order $b+1$.
For $\alpha\in \Delta_k(b), \beta\in \Lambda_k(b+1)$
we write
$$
Q_{\nu, \mu}^{\lambda, \alpha}=
\frac{\partial^\alpha q_{\mu, \lambda}}{\alpha!}
\left(w^{(0)}\right)
\left(w^{(\nu)}-w^{(0)}\right)^\alpha, \quad
Q_{\nu, \mu}^{\lambda, \beta}=
\frac{\partial^\beta q_{\mu, \lambda}}{\beta!}
\left(w^{(0)}\right)
\left(\zeta^{(\nu)}_{\lambda, \beta}-w^{(0)}\right)^\beta
$$
for some suitable intermediate point
$\zeta^{(\nu)}_{\lambda, \beta}$.
Then we have
$$
\Delta=\det\Big(\sum_{(\lambda,\alpha)}Q_{\mu, \nu}^{\lambda,\alpha}\Big),
$$
with the summation over
$(\lambda,\alpha)\in \{0,\ldots, N-1\}^n\times
\Delta_k(b+1)$.

We expand the determinant as previously
in terms of maps
$$
\tau: \Delta_n(d)\rightarrow
\{0,\ldots, N-1\}^n\times \Delta_k(b+1)
$$
giving
$$
\Delta=\sum_{\tau} \Delta_\tau,\quad \Delta_\tau=
\det\Big(Q^{\tau(\mu)}_{\mu, \nu}\Big).
$$
with the summation over $\tau$ as above.

Now if for some $\lambda\in M, k$ with $0\le k \le b$
we have
$$
\hash \tau^{-1}\big(\{\lambda\}\times\Lambda_m(k)\big)
> L_m(k)
$$
then $\Delta_\tau=0$ because
the corresponding columns
are dependent (the factors $(x^{(\nu)})^\lambda$
are constant on the rows in those columns).

Since there are $N^n$ possibilities for $\lambda$,
we have that
the total number of columns from which an expansion term
of degree $k$ may
be drawn for a surviving term is $N^nL_m(k)$.

{\bf We now assume that $rR<1$.\/}
Then every surviving term is estimated by
$$
\big[\big(n(b+1)^mC\big)^d\big]^{D_n(d)}\, (rR)^{B'}
$$
for some $B'\ge B=B(m,n,d, N)$, and since $rR<1$ every term
is estimated by the above with $B'=B$.

The total number of terms, assuming no cancellation, is
$$
D_n(d)!\left(N^n\,D_m(b+1)\right)^{D_n(d)}.
$$
Thus we have an integer $K\Delta$ with
$$
K|\Delta| \le H^{ndD_n(d)}\, D_n(d)!\left(N^nD_m(b+1)\right)^{D_n(d)}\,
\big[\big(n(b+1)^mC\big)^d\,\big]^D\, (rR)^{B}.
$$
And if $K|\Delta|\ge 1$ then so is its $B$th root.

Now we have (see \cite{Pila:Mild})
$$
L_m(d)=\left(\begin{array}{c}
m-1+d\\ m-1
\end{array}\right)
=
\frac{d^{m-1}}{(m-1)!}(1+o(1))
$$
where here and below $o(1)$ means as $d\rightarrow\infty$
for fixed $m, n, N$. Thus likewise
$$
D_m(d)=L_{m+1}(d)=\frac{d^{m}}{m!}(1+o(1)).
$$
We find that
$$
b(n,m,N,d)= \left(\frac{m!d^n}{n! N^n}\right)^{1/m}(1+o(1))
$$
for $d\rightarrow\infty$ with $n,m,N$ fixed. Thus
(by replacing the sum $\sum_{\delta=0}^bL_m(\delta)\delta$
by an integral)
$$
B(m,n,N,d)=E(m,n,N)d^{n(m+1)/m}(1+o(1))
$$
where $E$ is a suitable combinatorial expression.

With our choice of $d$ we have as before that
$H^{ndD_n(d)/B}\le E$ is bounded depending only on $n,m$
(once $H$ is sufficiently large in terms of $n,m$).
We also have
$$
\Big(D_n(d)!(N^nD_m(b+1))^{D_n(d)}\Big)^{1/B}=1+o(1)
$$
as $d\rightarrow\infty$, so is bounded by some $E$.

Finally, we have that
$$
\frac{dD}{B}=
\frac{E}{d^{n/m-1}}(1+o(1))\le
\frac{E}{d^{n/m-1}},
$$
(with a different $E$) while similarly
$$
\frac{bD}{B}=E(1+o(1))\le E.
$$

Therefore
$$
|K\Delta|^{1/B}\le E\,C\,Rr
$$
and all the points of $X(\QQ, H)$ in the image of the
disc lie on one hypersurface
of degree at most $d$ provided
$$
r<(C\, E\, R)^{-1}.
$$
The box $[-1,1]^m$ may be covered by
$$
C'=\left(c(n)CER+1\right)^m
$$
such discs, where $c(n)$ is the maximum side of a
cube inscribed in a unit $n$-sphere.
This gives the desired conclusion for $H\ge H_0(m,n,N)$
and for smaller $H$ it follows as the number of such
rational points is bounded. \qed

\subsection{Proof of 2.3.6 and 2.3.7}
In the case of 2.3.6, by Theorem 2.3.1, $X_t(\QQ, H)$ is contained
in the intersection of $X_t$ with at most
$C_1({\mathcal X})(\log H)^{c_1({\mathcal X})}$
hypersurfaces of degree $d=[(\log H)^{m/(n-m)}]$, while
in the case of 2.3.7, by Theorem 2.3.2, $X_t(\QQ, H)$ is contained
in the intersection of $X_t$ with at most $C_2({\mathcal X})$
hypersurfaces of degree $d=[(\log H)^{m/(n-m)}]$.

If any such intersection has dimension 2, then the Pfaffian
functions parameterizing the surface $X_t$ identically
satisfy some algebraic relation. Then the surface $X_t$
is algebraic, and $X_t^{\rm trans}$ is empty.

Thus we may assume that all the intersections have dimension
at most 1. We will treat these following the method in \cite{Pila:Mild}
by dividing the intersections into graphs of functions with suitable
properties, and estimating the rational points on any such graphs
which are not semi-algebraic using the Gabrielov-Vorobjov estimates.

Suppose that the fibre $X_t$ is the intersection of $[-1,1]^3$
with the Pfaffian surface defined by the Pfaffian functions
$$
x, y, z: G\rightarrow \RR
$$
of complexity (at most) $(r, \alpha, \beta)$.
Write $(p,q)$ for the variables in $G$. Suppose that the polynomial
$F(x,y,z)$ of degree at most $d$ defines the hypersurface $V=V_F$.

The intersection $X_t\cap V$ is the image of the one-dimensional
subset $W\subset G$ defined by
$$
\phi(p,q)=F(x(p,q), y(p,q), z(p,q))=0.
$$
It is thus the zero-set of a Pfaffian function of complexity $(r, \alpha, d\beta)$.
The singular set $W_s\subset W$ is defined by $\phi=\phi_p=\phi_q=0$,
the zero-set  Pfaffian functions of complexity $(r, \alpha, \alpha+d\beta-1)$
(see \cite[2.5]{GV}).

At a point of $W-W_s$, $W$ is locally the graph of a real-analytic function
parameterized by $p$ if $\phi_p\ne 0$, or $q$ if $\phi_q\ne 0$.

Proceeding as in \cite{Pila:Mild}, we decompose $V_F$ into ``good'' curves, and
points. Here a ``good'' curve is a connected subset whose projection
into each coordinate plane of $\RR^3$ is a ``good'' graph with respect
to one or other of the axes; namely, the graph of a function $\psi$ which
is real analytic on an interval, has slope of absolute value at most 1 at every point, and such
that the derivative of $\psi$ of each order $1,\ldots, [\log H]$ is either
non-vanishing in the interior of the interval or identically zero.

Using the topological estimates of Gabrielov-Vorobjov \cite[3.3]{GV}, Zell \cite{Zell},
and estimates for the complexities of the various Pfaffian functions involved
as in \cite{Pila:Mild}, one shows that $V_F$ decomposes into a union of at most
$$
C_5(r, \alpha, \beta) d^{C_6(r, \alpha, \beta)}
$$
points and ``good'' curves $Y$. If such a ``good'' curve  is semi-algebraic, then
so are its projections to each coordinate plane, and also conversely.
On a non-algebraic plane ``good'' graph $Y$, one has
$$
N(Y, H)\le C_7(r,\alpha, \beta) (d\log H)^{C_8(r, \alpha, \beta)}
$$
as in \cite{Pila:Mild}, using \cite{P2} and estimates
for Pfaffian complexity. Combining the last two estimates with
those in the first paragraph of the proof gives the
required conclusions.\ \qed

\section{Proof of the $C^r$-parameterization theorem} \label{sec:Proof:C^r}

In this section we prove Theorem 2.1.3, in a self-contained way except for Miller's preparation result \cite[Main Theorem]{MillerD}.
We also show a so-called pre-parameterization result for definable sets in $\Ranp$, used to generate $C^r$-parameterizations with the required number of maps essentially by composing with power maps (see Theorem \ref{pre-param} and how it is used to prove Theorem 2.1.3).
In Theorem \ref{pre-param}, there is at first sight no family version or parameter dependence, but, at second sight one sees that it is implicitly built in via a triangular property of the maps involved.

In Section \ref{sec:comp} we give some results about derivatives of compositions, related to mild functions and Gevrey functions, and we introduce the notion of weakly mild functions (see Definition \ref{def:weaklymild}). In Section \ref{sec:a-b-m} we define a-b-m functions, and relate them with weakly mild functions. In Section \ref{sec:first:param} we state the pre-parameterization result and use it to prove Theorem 2.1.3.  In Section \ref{sec:preparation:Miller} we prove our pre-parameterization result using a preparation result from \cite{MillerD}.

\subsection{Compositions}\label{sec:comp}

We first equip the notion of mild functions from \cite{Pila:Mild} with an order. Next, we introduce the related notion of weakly mild functions (see Definition \ref{def:weaklymild}). These notions are variants of the notion of Gevrey functions \cite{Gevrey}.

\begin{defn}\label{def:mild}
Let $A>0$ and $C\geq 0$ be real, and let $r>0$ be either an integer or $+\infty$.
A function $f:U\subset (0,1)^m\to [-1,1]$ with $U$ open is called $(A,C)$-mild up to order $r$ if it is $C^r$ and for all $\alpha\in\NN^m$ with $|\alpha|\leq r$ and all $x\in U$ one has
$$
| f^{(\alpha)}(x) |  \leq \alpha! (A|\alpha|^C)^{|\alpha|},
$$
where $\alpha!=\prod_{i=1}^{m} \alpha_i!$ and $|\alpha|=\sum_{i=1}^{m} \alpha_i$ as usual.
Call a map $f:U\subset (0,1)^m\to [-1,1]^n$ $(A,C)$-mild up to order $r$ if all of its component functions are.
\end{defn}

\begin{defn}[Weakly mild functions]\label{def:weaklymild}
Let $A>0$ and $C\geq 0$ be real, and $r>0$ be either an integer or $+\infty$.
A function $f:U\subset (0,1)^m\to [-1,1]$ with $U$ open is called weakly $(A,C)$-mild up to order $r$ if it is $C^r$ and for all $\alpha\in\NN^m$ with $|\alpha|\leq r$ and all $x\in U$ one has
$$
| f^{(\alpha)}(x) |  \leq \frac{ \alpha! (A|\alpha|^C)^{|\alpha|} } {x^\alpha} ,
$$
where $x^\alpha$ stands for $\prod_{j=1}^m x_j^{\alpha_j}$.
Call a map $f:U\subset (0,1)^m\to [-1,1]^n$ weakly $(A,C)$-mild up to order $r$ if all of its component functions are.
\end{defn}

We say mild (resp.~weakly mild) for $(A,C)$-mild (resp.~weakly $(A,C)$-mild) up to order $+\infty$ for some $A>0$ and some $C\geq 0$.
By the theory of Gevrey functions \cite{Gevrey}, it is known that a composition of mild functions is mild. Here we study some related results about compositions, with proofs based on Fa\`a di Bruno's formula.  (We do no effort to control the bounds beyond what we need.)
The next lemma is obvious by the chain rule for derivation.
\begin{lem}\label{rem:mild}
Let $r>0$ be an integer and let
$$
f:U\subset (0,1)^m\to [-1,1]
$$
be $(A,C)$-mild up to order $r$.
Then, for any $\xi\in (0,1)^m$, the function $$
V\to [-1,1]:x\mapsto f(\xi + x/Ar^{C+1})
$$
has $C^r$-norm bounded by $1$,
where $V\subset (0,1)^m$ is the open set consisting of $x$ such that $\xi + x/Ar^{C+1} := (\xi_1+x_1/Ar^{C+1},\ldots,\xi_m+ x_m/Ar^{C+1})$ lies in $U$.
\end{lem}


Theorem 2.1 of \cite{ConstantineSavits} is a multi-variate form of Fa\`a di Bruno's formula for iterated derivatives of compositions, which we now recall.

\begin{prop}[\cite{ConstantineSavits}, Theorem 2.1]\label{prop:multi-bruno}
Let $m\geq 1$ and $d\geq 1$ be integers. 
Consider a composition $h=f\circ g$, with $g:U\subset \RR^d\to V\subset \RR^m$, $f:V\to\RR$ and $U$ and $V$ open.
Let $\nu\in\NN^d$ be a nonzero multi-index. Write $|\nu| =  n$, and suppose that $f$ and $g$ are $C^n$.
Then
$$
h^{(\nu)}
$$
is equal to the sum over $(\lambda,s,k,\ell)\in I$ of
the terms
\begin{equation}\label{eq:brunoterm:hnu}
a_{\nu,s,k, \ell} f^{(\lambda)} \prod_{j=1}^s \big( g^{(\ell_j)} \big)^{k_j}
\end{equation}
with
\begin{equation}\label{eq:brunoterm:a}
a_{\nu,s,k, \ell} = \frac{\nu!}{\prod_{j=1}^s (k_j!)(\ell_j!)^{|k_j|}},
\end{equation}
where $0^0=1$, $h^{(\nu)}$ and $g^{(\ell_j)}$ are evaluated at $u\in U$ and $f^{(\lambda)}$ at $g(u)$, and where $I$ consists of $(\lambda,s,k,\ell)$ with
$\lambda\in\NN^m$, $1\leq s \leq n$, $k_j\in\NN^m$, $\ell_j\in\NN^d$, $j=1,\ldots,s$, $0< |\lambda|\leq n$, $0<|k_j|\leq n$, $0<|\ell_j|\leq n$, $\ell_j \prec \ell_{j+1}$,
\begin{equation}\label{eq:brunocond}
\sum_{j=1}^s k_j = \lambda,\  \mbox{ and }  \sum_{j=1}^s |k_j|\ell_j = \nu.
\end{equation}
Here, $\ell_j \prec \ell_{j+1}$ for $j<s$ means that either $|\ell_j| < |\ell_{j+1}|$, or, $|\ell_j| = |\ell_{j+1}|$ and $\ell_j$ comes lexicographically before $\ell_{j+1}$. (For $j=s$, $\ell_j \prec \ell_{j+1}$ is no condition.)
Moreover, there exist $A>0$ and $C\geq 0$, depending only on $m$ and $d$, such that
\begin{equation}\label{eq:brunoterm:sum}
\sum_{(\lambda,s,k,\ell) \in I} a_{\nu,s,k, \ell}   \leq (An^{C})^{n}.
\end{equation}
\end{prop}
\begin{proof}
We only have to prove (\ref{eq:brunoterm:sum}), since the other part is literally Theorem 2.1 of \cite{ConstantineSavits}.
But (\ref{eq:brunoterm:sum}) follows from
$$
a_{\nu,s,k, \ell} = \frac{\nu!}{\prod_{j=1}^s k_j! (\ell_j!)^{|k_j|}}\leq \nu! \leq n^n
$$
and
$$
\# I \leq  (n+1)^{1 + m + mn + nd },
$$
where the latter bound is obtained by letting the $s$, $\lambda_i$, $k_{ji}$ and $\ell_{jr}$ for $i=1,\ldots,m$,  $j=1,\ldots,n$, and $r=1,\ldots,d$ run independently from $0$ to $n$ when estimating the number of elements of $I$.
\end{proof}

The main purpose of our notion of weakly mild functions is that composition with $r$th power maps makes them mild up to order $r$, as follows.

\begin{prop}[Composition with power maps]\label{ell}
Let $A>0$ and $C\geq 0$ be real numbers and let $m>0$ be an integer.
Let $f:V\to [-1,1]$ be a function on some open $V\subset (0,1)^m$.
Assume for each $\beta\in \NN^m$ with $|\beta|\leq 1$ that $f^{(\beta)}$ is weakly $(A,C)$-mild up to order $+\infty$.
Then, there is $(A',C')$, depending only on $m$, $A$ and $C$, such that for any integers $r>0$ and $L_i\geq r$,
the composition $h=f\circ g$ of $f$ with
$$
g :x\mapsto x^L := (x_1^{L_1},\ldots,x_m^{L_m})
$$
on the open $U\subset (0,1)^m$ consisting of $x$ with $x^L\in V$,
is $(MA',C')$-mild up to order $r$, with $M=\max_i L_i$.
\end{prop}
\begin{proof}
By Proposition \ref{prop:multi-bruno} with $d=m$, we only have to estimate
$$
|f^{(\lambda)} \prod_{j=1}^s \big( g^{(\ell_j)} \big)^{k_j}|
$$
for $\nu$ with $1\leq |\nu|\leq r$ and $(\lambda,s,k,\ell)$ as in Proposition \ref{prop:multi-bruno}.
Fix $\nu $ with $1\leq |\nu|\leq r$ and write $|\nu|=n$.  If $n=1$ the statement follows from the conditions on $f^{(\beta)}$ for $\beta$ with $|\beta|\leq 1$. So let us suppose $n>1$. Fix $s$ with $1\leq s\leq n$ and $\lambda\in \NN^m$ with $|\lambda| \leq n$. Fix $x$ in $U$.
Choose $\lambda'$ and $\beta$ in $\NN^m$ with $\lambda' + \beta = \lambda$ and $|\beta|=1$. (Near the end we will optimize the choice of $\beta$, depending on $x$.)
By the weak $(A,C)$-mildness of $f^{(\beta)}$, we have
$$
|f^{(\lambda)}(x^L)|  = |( f^{(\beta)})^{(\lambda')} (x^L) | \leq \frac{ \lambda'! (A|\lambda'|^C)^{|\lambda'|} } {x^{L\lambda'}},
$$
where
$$
L\lambda' = (L_i\lambda'_i)_i.
$$
We further estimate
$$
|\prod_{j=1}^s \big( g^{(\ell_j)}(x) \big)^{k_j}| \leq |\prod_{j=1}^s M^{|\ell_j|\cdot|k_j|} x^{(L-\ell_j)k_j} |
$$
where $k_j$ and $\ell_j$ are as in (\ref{eq:brunocond}) and where
$$
(L-\ell_j)k_j = ((L_i-\ell_{ji})k_{ji} )_i.
$$
By (\ref{eq:brunocond}) it follows that
\begin{eqnarray*}
\frac{\prod_{j=1}^s M^{|\ell_j|\cdot|k_j|}  x^{(L-\ell_j)k_j} }{x^{L\lambda'}} & = & M^{\sum_{j=1}^s |\ell_j|\cdot|k_j|}  x^{-L\lambda' + \sum_{j=1}^s(L-\ell_j)k_j}\\
 & \leq & M^{n} x^{L\beta -  \sum_{j=1}^s \ell_jk_j},
\end{eqnarray*}
where $L\beta=(L_i\beta_i)_i$ and similarly $\ell_jk_j=(\ell_{ji}k_{ji})_i$.
Since this last inequality holds for any choice of $\beta$ with $|\beta|=1$ (and the corresponding $\lambda'$), let us choose $\beta$ with $\beta_{i_1}=1$ where $i_1$ is such that $x_{i_1}=\min_i x_i$, where the minimum is over $i$ with $\lambda_i>0$.
Now one has $|\sum_{j=1}^s \ell_jk_j |\leq n$ by (\ref{eq:brunocond}) and $|L\beta|=L_{i_1} \geq r \geq n$ and thus
$$
x^{L\beta -  \sum_{j=1}^s \ell_jk_j} \leq 1.
$$
Putting together we find
$$
|f^{(\lambda)} \prod_{j=1}^s \big( g^{(\ell_j)} \big)^{k_j}|\leq \lambda'! (A|\lambda'|^C)^{|\lambda'|}M^{n} \leq \nu ! (MAn^C)^{n},
$$
which finishes the proof.
\end{proof}

\subsection{a-b-m functions}\label{sec:a-b-m}

We introduce the notions of bounded-monomial functions and of a-b-m functions in Definition \ref{defn:b-m}, resp.~\ref{defn:prepared}.
Our motivation for the notion of a-b-m functions is threefold: Firstly, a-b-m functions behave well in the sense that they are weakly mild whenever valued in $[-1,1]$ (see Proposition \ref{prop:abm-weakly-mild}), and even more holds (see Corollary \ref{cor:abm-C1-weakly-mild}). Secondly, even better so than weakly mild functions, a-b-m functions can be rendered mild up to order $r$ by composing with power maps, see Proposition \ref{prop:pre-param-cell} (where the occurring roots behave better for a-b-m functions than for weakly mild ones). Thirdly, using a preparation result from Miller \cite{MillerD}, we can rather easily obtain parameterizations with a-b-m maps. Some more work is needed to make the parameterizations even better (essentially with some extra control on the first partial derivatives as in Theorem \ref{pre-param} \ref{pre4}), so that it can be combined with the power maps result \ref{prop:pre-param-cell}.


\begin{defn}[bounded-monomial functions]\label{defn:b-m}
Let $U$ be a subset of $(0,1)^m$.
A function $b:U\to \RR$ is called bounded-monomial if it has bounded range and it is
of the form
$$
b(x) = x^\mu := \prod_{i=1}^m x_i^{\mu_{i}}
$$
for some $\mu_{i}$ in $\RR$. A map $U\to \RR^n$ is called bounded-monomial if all of its component functions are.
\end{defn}

\begin{defn}[a-b-m functions]\label{defn:prepared}
Let $U$ be a subset of $(0,1)^m$.
A function $f:U\to \RR$ is called a-b-m, in full analytic-bounded-monomial, if it is of the form
$$
f(x) = b_j(x) F( b_1(x),\ldots, b_s(x) )
$$
for some bounded-monomial map $b:U\to \RR^s$ for some $s$ and for some nonvanishing analytic function $F:V\to \RR$, where $V$ is an open neighborhood of $\overline{b(U)}$, the topological closure of $b(U)$ in $\RR^s$, and where $j$ lies in $\{1,\ldots,s\}$. We call the map $b$ an associated bounded-monomial map of $f$.

Finally, call a map $f:U\to \RR^n$ a-b-m, with associated bounded-monomial map $b$, if all its component functions are (namely, each $f_i$ is a-b-m, and, $b$ is an associated bounded-monomial map for each $f_i$).
\end{defn}

\begin{prop}\label{prop:abm-weakly-mild}
Let $h:U\subset (0,1)^m\to \RR$ be an a-b-m function. Suppose that $U$ is open and that $h(U)\subset [-1,1]$. Then the function $h$ is weakly $(A,C)$-mild up to order $+\infty$ for some $A>0$ and some $C\geq 0$.
\end{prop}
\begin{proof}
Any function $F:V\to [-1,1]$ on an open $V$ in $(0,1)^n$ such that $F$ is analytic on some open neighborhood of $\overline V$ (the topological closure of $V$ in $\RR^n$) is $(A_0,0)$-mild up to order $+\infty$ for some $A_0>0$, see e.g.~\cite{Gevrey}.
Also, for any real $S\geq 1$ and any bounded-monomial function $b:U\subset (0,1)^m\to (0,S) : x\mapsto x^\mu$ with $U$ open in $(0,1)^m$, the function $b/S$ is weakly $(A_1,C_1)$-mild up to order $+\infty$ for some $A_1>0$ and $C_1\geq 0$. One may for example take $A_1= t^{t}$ with $t=\max(2, |\mu_1|,\ldots,|\mu_m|)$ and $C_1=0$.
The lemma now follows from the fact that a composition $f\circ g$ is automatically weakly $(A'',C'')$-mild up to order $r$, whenever $f:V\to \RR$ is $(A,C)$-mild up to order $r$ and $g:U\to V$ is weakly $(A',C')$-mild up to order $r$, with open sets $U\subset (0,1)^d$ and $V\subset (0,1)^m$, and where moreover $A''$ and $C''$ depend only on $A,A',C,C',m,d$.
This fact is easy to see as follows. As in the beginning of the proof for Proposition \ref{ell}, by Proposition \ref{prop:multi-bruno} we only have to estimate a single term of the form
$$
|f^{(\lambda)} \prod_{j=1}^s \big( g^{(\ell_j)} \big)^{k_j}|,
$$
for $(\lambda,s,k,\ell)$ as in Proposition \ref{prop:multi-bruno} and $\nu$ with $|\nu|\leq r$.
By the $(A,C)$-mildness of $f$ up to order $r$, and assuming $A\geq 1$, we have
$$
|f^{(\lambda)}(x)| \leq \lambda! (A|\lambda|^C)^{|\lambda|} \leq \nu! (A n^C)^{n}.
$$
By the weak $(A',C')$-mildness of $g$ up to order $r$ we have
$$
|\prod_{j=1}^s \big( g^{(\ell_j)}(x) \big)^{k_j}| \leq | \prod_{j=1}^s \frac{ \ell_j! (A' |\ell_j|^{C'} )^{|\ell_j|\cdot |k_j|} } { x^{\ell_j |k_j|}   }   |
$$
where $k_j$ and $\ell_j$ are as in Proposition \ref{prop:multi-bruno}.
Now by (\ref{eq:brunocond}) of Proposition \ref{prop:multi-bruno}, and assuming $A'\geq 1$, one has
$$
|\prod_{j=1}^s \big( g^{(\ell_j)}(x) \big)^{k_j}| \leq
|  \frac{ \nu! (A' n^{C'} )^{\sum_j |\ell_j|\cdot |k_j|} } { x^{\sum_j \ell_j |k_j|}   }   |
\leq |  \frac{ \nu! (A' n^{C'} )^{n} } { x^{\nu }   }   |.
$$
Putting together we find $A''$ and $C''$ as desired.
\end{proof}


The a-b-m functions with an associated bounded-monomial map $b$ such that moreover $b$ has bounded $C^1$-norm have particularly nice properties as illustrated by the next two results.

\begin{cor}\label{cor:abm-C1-weakly-mild}
Let $f:U\subset (0,1)^m\to \RR$ be an a-b-m function 
with an associated bounded-monomial map $b$ such that $b$ has bounded $C^1$-norm. Then, for each $j=1,\ldots,m$,  the function $\partial f/\partial x_j$ is a finite sum of a-b-m functions on $U$. Hence, there is $\varepsilon>0$ such that the functions $\varepsilon f$ and $\varepsilon \partial f/\partial x_j$ are weakly $(A,C)$-mild up to order $+\infty$ for some $A>0$ and some $C\geq 0$.
\end{cor}
\begin{proof}
The first statement follows from the definition of a-b-m functions, the chain rule for derivation, and since the $C^1$-norm of $b$ is bounded. In detail,
write $f=b_{j_0}F(b)$ as in Definition \ref{defn:prepared} where the map $b:U\to\RR^s$ has bounded $C^1$-norm and $j_0\in\{1,\ldots,s\}$. Then
$$
\frac{\partial f}{\partial x_j}
$$
equals the sum of
$$
\frac{\partial b_{j_0}}{\partial x_j} F(b)
\mbox{ and the }
b_{j_0}\frac{\partial F(b)}{\partial b_\ell}  \frac{\partial b_\ell}{\partial x_j}
\mbox{ for }
\ell=1,\ldots,s.
$$
Clearly $(\partial b_{j_0}/\partial x_j) F(b)$  is a-b-m. For the other terms, one takes real $S$ such that $|(\partial F/\partial b_\ell)(b)|+1<S$ on $U$, and, one rewrites $\partial F/\partial b_\ell(b)$ as the sum of $S+\partial F/\partial b_\ell(b)$ with $-S$. Plugging this in gives $b_{j_0}(\partial F/\partial b_\ell(b) (\partial b_\ell/\partial x_j)$ as a sum of two terms which are both a-b-m.
The final part follows easily from Proposition \ref{prop:abm-weakly-mild}.
\end{proof}


Similarly to Proposition \ref{ell}, cylindrical sets and their walls are rendered mild up to order $r$ after suitable composition with power maps, when the initial walls are nice enough, see the following definitions and Proposition \ref{prop:pre-param-cell}. Here, a-b-m functions show their essential use.

\begin{defn}[Cylindrical sets and their walls]
A subset $C\subset \RR^n$ is called a cylindrical set, if
$$
C = \{x\in \RR^n\mid \wedge_{i=1}^n\,  \alpha_i(x_{<i}) \sq_{i1} x_i \sq_{i2} \beta_i(x_{<i}) \}
$$
for some continuous functions $\alpha_i$ and $\beta_i$ with $\alpha_i < \beta_i$, $x_{<i} = (x_1,\ldots,x_{i-1})$, and with $\sq_{i1}$ either $=$, $<$, or no condition, and with $\sq_{i2}$ either $<$ or no condition, with the conventions that $\sq_{i2}$ is no condition if $\sq_{i1}$ is equality. If $\sq_{i1}$ is $=$ or $<$ then we call $\alpha_i$ a wall of $C$. Likewise, if $\sq_{i2}$ is $<$ then we also call $\beta_i$ a wall of $C$.
\end{defn}

\begin{prop}\label{prop:pre-param-cell}
Let $U\subset (0,1)^m$ be a cylindrical set which is moreover open in $(0,1)^m$. Suppose for each wall $\alpha$ of $U$ that $\alpha$ is a-b-m with an associated bounded-monomial map $b$ such that moreover $b$ has bounded $C^1$-norm. For any integer $r>0$, write $\psi_{r}$ for the map sending $x\in(0,1)^m$ to
$$
\psi_{r}(x) = (x_1^{r^m},x_2^{r^{m-1}},\ldots,x_m^{r}),
$$
and write $U_{r}$ for $\psi_{r}^{-1}(U)$.
Then there exist $A>0$ and $C\geq 0$, depending only on $U$, such that $U_{r}$ is an open cylindrical set whose walls are $(r^{m}A,C)$-mild up to order $r$ for each $r>0$.
\end{prop}
\begin{proof}
Let $\alpha$ be a wall of $U$, say, bounding the $i$-th variable from below. Then the corresponding wall $\alpha_r$ of $U_r$ (bounding the $i$-th variable from below) satisfies
\begin{equation}\label{eq:alphar}
\alpha_r(x_{<i}) 
= \sqrt[ r^{m - i + 1} ]{\alpha( x_1^{r^m},x_2^{r^{m-1}},\ldots,x_{i-1}^{r^{m-i+2}} )}.
\end{equation}
We show the existence of $A'>0$ and $C'\geq 0$ such that $\alpha_r$ is $(r^mA',C')$-mild up to order $r$, where $A'$ and $C'$ depend only on $U$.
If $\alpha>\varepsilon$ on $U$ for some $\varepsilon>0$, then the existence of $A'$ and $C'$ as desired follows easily from (\ref{eq:alphar}), Proposition \ref{ell} and the chain rule. If $\alpha$ is identically zero then so is $\alpha_r$ and $A'$ and $C'$ exist clearly. Since any a-b-m function $f$ is the product of a bounded-monomial function $b_j$ with an a-b-m function $F(b)$ with $|F(b)|>\varepsilon$ for some $\varepsilon>0$, and since products behave well for mildness up to order $r$, it is enough to show the existence of $A'>0$ and $C'\geq 0$  in the case that $\alpha$ is itself bounded-monomial. That is, we may suppose that there is $\mu$ in $\RR^{m}$ with
\begin{equation}\label{eq:alpha}
\alpha(z_{<i}) = z^{\mu}
\end{equation}
for $z\in U$.
For $k\geq 1$ let $V_k$ be the set consisting of $y\in (0,1)^m$ with $(y_1^k,\ldots,y_{m}^k)$ in $U$ and let $h_k$ be the function
$$
h_k:V_k\to \RR:y\mapsto y^\mu.
$$
We claim that there exist constants $S>0$, $A_0>0$ and $C_0\geq 0$, depending only on $U$ (and not on $k$), such that for each $\beta\in \NN^m$ with $|\beta|\leq 1$ the function
$$
\frac{ |h_k^{(\beta)}| }{S}
$$
is weakly $(A_0,C_0)$-mild.
By (\ref{eq:alphar}) and (\ref{eq:alpha}), $\alpha_r$ is the composition of $h_k$ with $x\mapsto x^L$ for some $k\geq 1$ and some $L$ in $\NN^m$ with $r\leq L_i\leq r^{m-1}$ for each $i$. Thus, the existence of $A'$ and $C'$ now follows from the claim and Proposition \ref{ell}.
There is only left to prove the claim to finish the proof.
Since both $\alpha$ and $h_k$ are monomials with exponent tuple $\mu$ (but with different domain), we find by deriving
\begin{equation}\label{eq:alphahk}
|h_k^{(\lambda)}(y)| = c(\lambda,k,\mu)|\alpha^{(\lambda)}(z)|^{1/k}
\end{equation}
with $z = (y_1^k,\ldots,y_{m}^k)$, and, with
\begin{equation}\label{eq:alphack}
c(\lambda,k,\mu)< \lambda! (A_1|\lambda|^{C_1})^{|\lambda|}
\end{equation}
for some $A_1>0$ and $C_1\geq 0$ depending only on $U$. The claim now follows from the bounds on $|\alpha^{(\lambda)}(z)|$ for $z\in U$ and $\lambda\in\NN^m$ that follow from the facts that $\alpha$ is bounded-monomial with bounded $C^1$-norm.
\end{proof}

\subsection{Pre-parameterization and the proof of Theorem 2.1.3}\label{sec:first:param}

Let us fix some terminology for the rest of Section \ref{sec:Proof:C^r} (similar to convention 2.1.1). We will work with certain reducts of $\Ranp$, following \cite{MillerD}.
For any language $\cL$ on $\RR$ and any subfield $K$ of $\RR$, let us denote by $\cL^K$ the expansion of $\cL$ by the functions
$$
x \mapsto \begin{cases} x^\rho, \mbox{ if } x> 0,\\ 0 \mbox{ otherwise, }\end{cases}
$$
for $\rho\in K$. Let $\cL_{an}$ be the subanalytic language (in particular, the $\cL_{an}$-definable subsets of $\RR^n$ are precisely the globally subanalytic subsets of $\RR^n$).
Let $\cF$ be a Weierstrass system and let $\cL_\cF$ be the corresponding language as in \cite[Definition 2.1]{MillerD}. (The language $\cL_\cF$ is always a reduct of $\cL_{an}$.) By the field of exponents of $\cF$ is meant the set of real $r$ such that $(0,1)\to\RR:x\mapsto (1+x)^r$ is $\cL_\cF$-definable; this set is moreover a field, see Remark 2.3.5 of \cite{MillerD}.
Let $K$ be a subfield of the field of exponents of $\cF$. From now on up to the end of Section \ref{sec:Proof:C^r} we will work with the $\cL_\cF^K$-structure on $\RR$, definable will mean $\cL_\cF^K$-definable, and, we say cell for a definable cylindrical set.
Note that $\cL_{an}^\RR$ is an example of such a language, whose structure on $\RR$ was denoted by $\Ranp$ above. 

We can now state our pre-parameterization result, that generates $C^r$-parameterizations essentially by composing with power maps, and where possible parameter dependence is built in via the triangularity property \ref{pre3}.

\begin{thm}[Pre-parameterization]\label{pre-param}
Let $X\subset (0,1)^{n}$ be definable, and suppose that $X$ is the graph of a definable function $f: U\to (0,1)^{n-m}$ for some $m\geq 0$ and open set $U\subset (0,1)^m$. Then there exist finitely many definable maps
$$
\varphi_i:U_i\to X
$$
such that the following hold
\begin{enumerate}[label = \textup{(\arabic*)}, ref = (\arabic*)]
\item\label{pre1} $\bigcup_i \varphi_i(U_i)=X$.

\item\label{pre2} Each $U_i$ is an open cell in $(0,1)^m$. 

\item\label{pre3} Each $\varphi_i$ is a triangular map, in the sense that for each $j\leq m$ there is a unique map $\Pi_{<j}(U_i)\to \Pi_{<j}(X)$ making a commutative diagram with $\varphi_i$ and the projection maps $X\to
    \Pi_{<j}(X)= \Pi_{<j}(U)$ and $U_i\to \Pi_{<j}(U_i)$, with in both cases $\Pi_{<j}$ the projection on the first $j-1$ coordinates.

\item\label{pre4} 
For each $i$, the map $\varphi_i$ and the walls $\alpha$ of $U_i$ 
are a-b-m with an associated bounded-monomial map $b_i$ such that $b_i$ has bounded $C^1$-norm. 
\end{enumerate}
\end{thm}

In a way, property \ref{pre4} is a key new property for parameterizations, and it may be compared with preparation results from \cite{NguyenValette} and Lipschitz continuity results from \cite{Kurdyka}, \cite{Pawlucki} and \cite{Valette2}.

Theorem 2.1.3 follows directly from Theorem \ref{pre-param}, Propositions \ref{ell}, \ref{prop:pre-param-cell}, Corollary \ref{cor:abm-C1-weakly-mild} and Lemma \ref{rem:mild},   as follows.

\begin{proof}[Proof of Theorem 2.1.3]
Up to finite partitioning and up to transforming $T$ if necessary, we may suppose that $X=\{(t,x)\mid t\in T,\ x\in X_t\}$ equals the graph of a function $f:U\subset T\times (0,1)^m\to (0,1)^{n-m}$ where $U$ is an open cell in $(0,1)^{k+m}$. Apply Theorem \ref{pre-param} to $X$ to find maps $\varphi_i$ on cells $U_i\subset (0,1)^{k+m}$.
For any integer $r>0$ 
write $U_{i,r}$ for the set of $(t,x)\in (0,1)^{k+m}$ such that
$$
\psi_r(t,x):=(t_1^{r^{k+m}},\ldots,t_k^{r^{1+m}},x_1^{r^m},\ldots,x_m^r)
$$
lies in $U_i$ and write $\varphi_{i,r}:U_{i,r}\to X$ for the composition of $\varphi_i$ with $\psi_{r}$. Now it follows from Propositions \ref{ell},  \ref{prop:pre-param-cell}, Corollary \ref{cor:abm-C1-weakly-mild}, Lemma \ref{rem:mild} and by transforming the $T$-space back into the original form using property \ref{pre3}, that there are $c$ and $d$, and for each $r$ definable functions $\phi_{r, i, t,0} : V_{r,i,t,0}\subset (0,1)^m\to X_t$ with $i=1,\ldots,cr^d$, such that $V_{r,i,t,0}$ is an open cell whose walls have $C^r$-norm at most $1$ and the functions $\phi_{r, i, t,0}$ have $C^r$-norm at most one. By finally composing with the obvious triangular map, namely, if the $i$th variable runs between the walls $\alpha_i$ and $\beta_i$ in $V_{i,r,t,0}$ then one composes with $$
(0,1)^m\to V_{r,i,t,0} : z \mapsto (\alpha_i + (\beta_i-\alpha_i)z_i)_i,
$$
one gets the maps $\phi_{r, i, t}$ with domain $(0,1)^m$ as desired by Theorem 2.1.3.
\end{proof}


The remainder of Section \ref{sec:Proof:C^r} is devoted to the proof of Theorem \ref{pre-param}.

\subsection{The preparation result from \cite{MillerD} and proof of Theorem \ref{pre-param}}\label{sec:preparation:Miller}

We recall the preparation result from \cite{MillerD} for definable functions, where definable means $\cL_\cF^K$-definable as convened in Section \ref{sec:first:param}.  This preparation result is crucial in the proof of Theorem \ref{pre-param}.

Let $C\subset \RR^n$ be a cell. Write $\Pi_{<n}:\RR^n\to\RR^{n-1}$ for the projection sending $x$ in $\RR^n$ to $x_{<n}
=(x_1,\ldots,x_{n-1})$.

\begin{defn}[Prepared with centre]\label{defn:prep:center}
Let $\theta:\Pi_{<n}(C)\to\RR$ be a continuous, definable function whose graph is either disjoint from $\overline C$, or, contained in $\overline C\setminus C$, where $\overline C$ stands for the topological closure of $C$ in $\RR^n$. Furthermore, suppose that either $\theta=0$ on $B$, or $\theta(x_{<n})\sim x_n$ for $x\in C$, meaning that there is $S>1$ such that $x_n/S\leq \theta(x_{<n})\leq S x_n$ for all $x\in C$. Then $\theta$ is called a centre for $C$.
A definable function $f:C\to\RR$ with bounded range is called prepared with centre $\theta$ if $f$ can be written as
$$
f(x) =  b_j(x) F( b(x) )
$$
with some nonvanishing analytic function $F:V\to \RR$, where $V$ is an open neighborhood of $\overline{b(C)}$, the topological closure of $b(C)$ in $\RR^s$, and where $b:C\to \RR^s$ is a map with bounded range whose component functions $b_i$ have the form
$$
x\mapsto a_i(x_{<n}) |x_n - \theta(x_{<n})|^{r_i}
$$
with $r_i\in K$ and where $a_i:\Pi_{<n}(C)\to\RR$ is definable, and where $j$ lies in $\{1,\ldots,s\}$.

We call the map $b$ an associated bounded range map for $f$.
Say that a definable map $f:C\to\RR^n$ for $n\geq 1$ is prepared with centre $\theta$, and associated bounded range map $b$, if all its component functions are (namely, each $f_i$ is prepared with centre $\theta$ and associated bounded range map $b$.) 
\end{defn}

Note that the a-b-m functions of Definition \ref{defn:prepared} are much more specific than the prepared functions with centre $\theta$ of Definition \ref{defn:prep:center}. Indeed, a-b-m functions are, in a way, prepared in each variable instead of in one variable only and moreover with centre(s) zero.

\begin{prop}[Preparation of definable functions, {\cite[Main Theorem]{MillerD}}]\label{prop:preparation:Miller}
Let  $f:X\subset \RR^n \to\RR^N$ be a definable map on a definable set $X$. Suppose that the range of $f$ is bounded.
Then there exists a finite partition of $X$ into definable cells $C_i$ with some centre $\theta_i$ such that $f_{|C_i}$ is prepared with centre $\theta_i$ for each $i$.
\end{prop}




We are now ready to prove Theorem \ref{pre-param}.

\begin{proof}[Proof of Theorem \ref{pre-param}]
We proceed by induction on $m$, the case $m=0$ being trivial. Suppose $m\geq 1$.
By o-minimal cell decomposition (see \cite[Chapter 3]{vdD}), we can easily find definable maps $\varphi_i:U_i\to X$ satisfying  \ref{pre1}, \ref{pre2} and \ref{pre3}.
By Proposition \ref{prop:preparation:Miller} and up to further finite partitioning, we may furthermore suppose for each $i$ that $U_i$ has a centre $\theta_i$ and that the maps $\varphi_i$
are prepared with centre $\theta_i$ and an associated bounded range map $b_i$.
We may even suppose that $\theta_i=0$ on $U_i$, up to translating in the $x_m$ variable. Indeed, by Definition \ref{defn:prep:center} there is $S>0$ such that $|\theta_i|+1<S$ on $U_i$, and we may suppose (up to further finite partitioning) that either $x_m - \theta_i>0$, or, $x_m - \theta_i<0$ on $U_i$. In the case that $x_m - \theta_i>0$ on $U_i$ (the other case is similar), one replaces $U_i$ by the cell
$$
\tilde U_i:=\{x\in (0,1)^m\mid  (x_{<m}, Sx_m  + \theta_i(x_{<m}) ) \in U_i \}
$$
and $\varphi_i$ by
$$
\tilde \varphi_i:\tilde U_i\to X: x\mapsto \varphi_i(x_{<m}, Sx_m  + \theta_i(x_{<m}) ).
$$
By a classical technique (with inverse functions) we will ensure moreover that $b_i$ is $C^1$ and that $|\partial b_{i,j}/\partial x_m|\leq 1$ for each component function $b_{i,j}$ of $b_i$. 
Up to partitioning $U_i$ into finitely many definable pieces and neglecting pieces of lower dimension by induction on $m$, we may suppose that $b_i$ is $C^1$ on $U_i$ and that there is $j$ such that $|\partial b_{i,j}/\partial x_m|$ is maximal on $U$, in the sense that $|\partial b_{i,j}(x)/\partial x_m|\geq |\partial b_{i,j'}(x)/\partial x_m|$ on $U$ for any $j'$. Similarly, for this $j$ we may furthermore suppose that either $|\partial b_{i,j}/\partial x_m|\leq 1$ on $U_i$, or, that $|\partial b_{i,j}/\partial x_m| >  1$ on $U_i$. In the first case, $U_i$ and $\varphi_i$ are as desired.
In the second case, we may, up to finite partitioning and using o-minimality as before, suppose that for each $x_{<m} = (x_1,\ldots,x_{m-1})$ the function sending $x_m$ to $b_{i,j}(x_{<m},x_m)$ is injective. 
Let $\tilde U_i$ be the image of $U_i$ under the map sending $x$ to $(x_{<m},b_{i,j}(x))$, and let $\tilde \varphi_i:\tilde U_i\to U_i$ be the inverse function of this map. Now $\tilde \varphi_i$ is as desired, by the chain rule. In particular, $\tilde \varphi_i$ is prepared with centre zero and associated bounded-range map $\tilde b_i$ with $|\partial \tilde b_{i,j}/\partial x_m|\leq 1$ for each component function $\tilde b_{i,j}$ of $\tilde b_i$.



\par
By the flexibility in Definition \ref{defn:prep:center} for choosing the associated bounded range map, we may from now on suppose that we have definable maps $\varphi_i$ satisfying \ref{pre1}, \ref{pre2}, \ref{pre3}, that the maps $\varphi_i$
are prepared with centre $0$ and associated bounded range map $b_i$ such that moreover
$$
|b_{i,j}|<1-\varepsilon,\mbox{ and } |\frac{\partial b_{i,j}}{\partial x_m}|<1-\varepsilon
$$
for each component function $b_{i,j}$ of $b_i$ and some $\varepsilon>0$.

\par
We will now construct a finite collection of definable maps
$$
\varphi_{i\ell}:U_{i\ell}\to X
$$
satisfying Properties \ref{pre1}, \ref{pre2}, \ref{pre3}, and \ref{pre4}, using the data we have so far. 
For 
each wall $\alpha$ of $U_i$ bounding $x_m$, and with $h$ being either $b_i$ or $\partial b_i / \partial x_m$, 
let $h_{\alpha}$ be the map
\begin{equation}\label{halpha}
h_{\alpha} :  \Pi_{<m}(U_{i})\to (-1,1)^s: x_{<m} \mapsto \lim_{x_m\to \alpha(x_{<m} ) } h(x_{<m},x_m),
\end{equation}
where $\Pi_{<m}(U_{i})$ is the image of $U_i$ under the coordinate projection $\Pi_{<m}$ sending $x$ to $x_{<m}$.
This limit always exists by Definition \ref{defn:prep:center} and the form of $h_{\alpha}$.
Let $\cG_i$ be the collection of functions on $\Pi_{<m}(U_{i})$ consisting of 
the maps $h_\alpha$ from (\ref{halpha}) and the walls $\alpha$ of $U_i$ bounding $x_m$.
Up to partitioning $U_i$ further we may suppose for each $g\in\cG_i$ that either $g$ is constant, or, $0<|g|<1$ on $U_i$ (or both).  Consider the map $F_i$ whose component functions are the maps $|g|$ for those $g$ in $\cG_i$ which are nonconstant.
Apply the induction hypothesis to the map $F_i$ instead of $f$ 
to find a finite collection of maps $\psi_{i\ell }:V_{i\ell }\to \textrm{Graph}(F_i)$ satisfying properties \ref{pre1}, \ref{pre2}, \ref{pre3}, and \ref{pre4}, with $\textrm{Graph}(F_i)$ in the role of $X$, and with associated bounded-monomial maps  $c_{i\ell}$ .
Plugging the newly obtained maps $\psi_{i\ell }$ in the previously obtained maps $\varphi_i$ we get Properties \ref{pre1}, \ref{pre2}, \ref{pre3}, and \ref{pre4} for $X$. With more detail, let $U_{i\ell }$ be the cell
$$
\{x\in (0,1)^m \mid (\psi_{i\ell }( x_{<m} )_{<m},x_m) \in U_i \}
$$
and let $\varphi_{i\ell }:U_{i\ell } \to X $ be the map 
$$
x\mapsto \varphi_i(\psi_{i\ell }( x_{<m} )_{<m} , x_m  ).
$$
By the above application of the induction hypothesis the function
$$
b_i\circ \varphi_{i\ell}
$$
is a-b-m with an associated bounded-monomial map $d_{i\ell}$ with bounded $C^1$-norm. Let $b_{i\ell}$ be the map  $(c_{i\ell},d_{i\ell})$.
Then the maps $\varphi_{i\ell }$ satisfy \ref{pre1}, \ref{pre2}, \ref{pre3} and \ref{pre4} with associated bounded-monomial maps $b_{i\ell }$. 
This finishes the proof of Theorem \ref{pre-param}.
\end{proof}

\begin{rem}
To address the mentioned question of \cite[below Remark 3.8]{Burguet-Liao-Yang} it seems useful to bound the number of maps $\varphi_i$ in Theorem \ref{pre-param} (in the case that $X$ is semi-algebraic) in terms of the semi-algebraic complexity of $X$. It also seems interesting to control the degree $d$ of the polynomial $cr^d$ of Theorem 2.1.3 in terms of $m$.
Also note that the maps $\phi_{r,i,t}$ of Theorem 2.1.3 are analytic on their domain for each $r,i,t$ (by their construction based on a-b-m maps), but, in general, they may not extend to an analytic map on an open neighborhood of the closed box $[0,1]^m$ (indeed, a-b-m maps don't extend in general).
\end{rem}

\section{The proof of the quasi-parameterization theorem}

In this section we prove Theorem 2.2.3. First of all, however, we require some general results about definable families
of holomorphic functions. The definability here is with respect to an {\it arbitrary} polynomially bounded
o-minimal expansion of the real field, which we now fix. Let us recall the following definition from subsection 2.2, from where
we also recall that $\Delta (R)$ denotes the (open) disc in $\C$ of radius $R$ and centred at the origin.

\vspace{3mm}

\noindent
{\bf Definition 5.1}  \hspace{1mm} A definable family $\Lambda = \{F_t : t \in T \}$ is called an $(R, m, K)$-family, where $R$, $K$ are positive real numbers and
$m$ is a positive integer, if for each $t \in T$, the function $F_t : \Delta (R)^m \to \C$ is holomorphic and for all $z \in \Delta(R)^m$,
$|F_t (z)| \leq K$.

\vspace{3mm}

Let us first observe that for $\Lambda$ such a family it follows from the Cauchy inequalities  that we have the following bounds
on the Taylor coefficients of each $F_t$ at $0 \in \C^m$
(where $\alpha! := \alpha_1 ! \cdots \alpha_m !$ for $\alpha = \langle \alpha_1 , \ldots, \alpha_m \rangle \in \N^m$).

\vspace{2mm}

\noindent 5.1.1 \hspace{1mm} For all $\alpha \in \N^m$ and all $t \in T$, \hspace{2mm} $\aF
\leq \frac{K}{R^{|\alpha|}}$.

\vspace{2mm}

(For all general results from the theory
of functions of several complex variables we refer the reader to the first chapter of \cite{Herve}.)

In particular, if $R>1$ then $\aF \to 0$ as $|\alpha| \to \infty$  and so for each $t \in T$ there exists some $M_t \in \N$
such that

\vspace{2mm}

\noindent 5.1.2 \hspace{1mm} for all $\alpha \in \N^m$, $\aF \leq$ max$\{ \aF : \alpha \in \N^m , |\alpha| \leq M_t \}$.

\vspace{2mm}

The crucial uniformity result, from which the quasi-parameterization theorem will follow, is that $M_t$ may be chosen to be independent
of $t$. This in turn will follow from the maximum modulus theorem and the following general result.

\vspace{3mm}

\noindent {\bf Lemma 5.2}  \hspace{1mm} {\it Let $1<r<R$, $0< \lambda \leq \h$ and let $\{ \theta_t : t \in T \}$ be a definable family of functions from $(0, R)$ to $(0, R)$.
Then there exists $\epsilon \in (0, \lambda)$ such that for all $t \in T$, there exists $y_t \in (r, R)$ such that
$\theta_t (y_t - \epsilon ) \geq \frac{1}{2} \theta_t (y_t )$.}

\begin{proof}

Suppose not. Then there exists a function $\eta : (0, \lambda) \to T$, which by the principle of definable choice
we may take to be definable, such that

\vspace{2mm}

\noindent 5.2.1 \hspace{1mm} for all $x \in (0, \lambda)$ and all $y \in (r, R)$, \hspace{1mm} $\theta_{\eta (x)} (y-x) < \h \theta_{\eta (x)} (y)$.

\vspace{2mm}

Pick some $\gamma \in (r, R)$ and consider the definable function $x \mapsto \theta_{\eta (x)} (\gamma)$ for $x \in (0, \lambda)$.
It follows from polynomial boundedness that there exist a positive integer $N$ and $\nu \in (0, \lambda )$ such that

\vspace{2mm}

\noindent 5.2.2 \hspace{1mm} for all $x \in (0, \nu )$, \hspace{1mm}  $\theta_{\eta (x)} (\gamma) > x^{N}$.

\vspace{2mm}

Now let $k$ be a  positive integer and set $x_0 := \frac{R-\gamma}{2k}$, so that $x_0 \in (0, \nu)$ for large enough $k$.
By applying 5.2.1 successively with $x=x_0$ and $y= \gamma + x_0, \ldots, \gamma + kx_0$ we see that
$0 < \theta_{\eta (x_0)} (\gamma) < \h \theta_{\eta (x_0)} (\gamma + x_0) < \cdots < (\h)^k \theta_{\eta (x_0 )}(\gamma + kx_0 ) < (\h)^k R$.

So by 5.2.2 we obtain $(\frac{R-\gamma}{2k})^N = x_0^N < \theta_{\eta (x_0 )}(\gamma) <  (\h)^k R$, which is the required contradiction
if $k$ is sufficiently large.
\end{proof}

\noindent {\bf Theorem 5.3}  \hspace{1mm} {\it Let $\Lambda = \{F_t : t \in T \}$ be an $(R, m, K)$-family with $R>1$.
Then there exists $M = M(\Lambda) \in \N$ such that for all $t \in T$
and all $\alpha \in \N^m$,}
 $${\it \aF \leq \text{max} \{ \aF : \alpha \in \N^m, |\alpha| \leq M \}.}$$

\begin{proof}

Since the conclusion is trivially true for those $t \in T$ such that $F_t \equiv 0$ (no matter how $M$ is chosen)
we may assume that no $F_t$ is identically zero.

For $t \in T$ define $\theta_t : (0, R) \to (0, R)$ by

\vspace{2mm}

\noindent 5.3.1 \hspace{1mm} $\theta_t (y) :=  \frac{R}{K} \text{sup} \{|F_t (z)| : z \in \Delta (y)^m \}.$

\vspace{2mm}

Now let $r := R^{\frac{3}{4}}$, $\lambda = \text{min}\{\h , R^{\frac{3}{4}} - R^{\frac{2}{3}} \}$ and apply 5.2 to obtain $\epsilon \in (0, \lambda)$ such that

\vspace{2mm}

\noindent 5.3.2 \hspace{1mm} for all $t \in T$, there exists $y_t \in (r, R)$ such that $\theta_t (y_t - \epsilon ) \geq \h \theta_t (y_t)$.

\vspace{2mm}

Now choose $D \in \N$ so that

\vspace{2mm}

\noindent 5.3.3 \hspace{1mm} $(1-\frac{\epsilon}{R} )^D < \frac{1}{4}$, and

\vspace{2mm}

\noindent 5.3.4 \hspace{1mm} $5(D+1)^m \leq 2R^{\frac{D}{3}}$.

\vspace{2mm}

We show that $M:= 2D$ satisfies the required conclusion. So fix some arbitrary $t \in T$ and let
$B_t := \text{max}\{ \aF :  |\alpha| \leq D \}$. (The $\alpha$'s range over $\N^m$ for the rest of this proof.)
It is clearly sufficient to show that

\vspace{2mm}

\noindent(*) \hspace{5mm}  for all $\alpha$ with $|\alpha| \geq 2D$ we have $\aF \leq B_t$.

\vspace{2mm}

To this end we consider the truncated Taylor expansion of $F_t$, namely $P_t (z) := \sum_{|\alpha| \leq D} \frac{F_t^{(\alpha)} (0)}{\alpha !} z^{\alpha}$.
Clearly we have

\vspace{2mm}

\noindent 5.3.5   \hspace{1mm} for all $z \in \Delta(R)^m$, \hspace{1mm} $|P_t (z)| \leq  (D+1)^m B_t R^D$.

\vspace{2mm}

Now choose $w = \langle w_1 , \ldots , w_m \rangle \in \Delta(R)^m$ with $|w_i| = y_t - \epsilon$ and $|F_t(w)| = \frac{K}{R} \theta_t (y_t - \epsilon)$ (which
is possible by 5.3.1 and the maximum modulus theorem), and let $\eta_i := \frac{w_i}{y_t - \epsilon}$ so that $|\eta_i| = 1$ for
$i = 1, \ldots , m$. Consider the function $H_t : \Delta(R) \to \C$ given by
$$H_t (u) := \frac{F_t(u\eta ) - P_t (u\eta)}{u^{D+1}}.$$
This is clearly a well-defined analytic function and by the maximum modulus theorem there exists $u_t \in \Delta (R)$ with $|u_t| = y_t$
such that $|H_t (y_t - \epsilon)| \leq |H_t (u_t )|$. Thus
$$|F_t((y_t-\epsilon)\eta) - P_t((y_t-\epsilon)\eta)| \leq (\frac{y_t-\epsilon}{y_t})^{D+1} |F_t((u_t\eta) - P_t((u_t\eta)|.$$
However, by 5.3.3, $(\frac{y_t-\epsilon}{y_t})^{D+1} \leq (1-\frac{\epsilon}{R} )^{D+1} < \frac{1}{4}$  so, upon recalling
that $w = (y_t - \epsilon)\eta$ and using 5.3.5, we see that

\vspace{2mm}

\noindent 5.3.6 \hspace{1mm} $|F_t(w)|  \leq  \frac{1}{4}|F_t(u_t \eta)| + \frac{5}{4}(D+1)^m B_t R^D$.

\vspace{2mm}

But $|F_t (w)| = \frac{K}{R} \theta_t (y_t - \epsilon) \geq \frac{K}{2R} \theta_t (y_t)$ by 5.3.2, and since $|u_t\eta| = y_t$ we obtain from
this and 5.3.1 that $|F_t (w)| \geq \h |F_t (u_t\eta)|$. Putting this into 5.3.6 we obtain

\vspace{2mm}

\noindent 5.3.7 \hspace{1mm} $|F_t(w)|  \leq  \frac{5}{2}(D+1)^m B_t R^D$.

\vspace{2mm}

Now, with a view to proving (*), let $|\alpha| \geq 2D$. Then by applying the Cauchy inequalities in the polydisk
$\Delta(y_t - \epsilon)^m$ and using 5.3.1 and 5.3.7 we obtain
$$\aF \leq \frac{K}{R} \theta_t (y_t - \epsilon) \cdot (y_t - \epsilon)^{-|\alpha|}
= |F_t (w)|(y_t - \epsilon)^{-|\alpha|} \leq \frac{5}{2}(D+1)^m B_t R^D(y_t - \epsilon)^{-|\alpha|}.$$

But $y_t - \epsilon \geq r-\lambda \geq R^{\frac{2}{3}}$ (by the definitions of $r$ and $\lambda$), so by 5.3.4

$$\aF \leq \frac{5}{2}(D+1)^mB_tR^D \cdot (R^{\frac{2}{3}})^{-2D} = \frac{5}{2}(D+1)^m R^{\frac{-D}{3}} B_t \leq B_t$$
as required.
\end{proof}

It follows immediately from 5.3 that the function $\kappa_{\Lambda} : T \to \R$ given by

\vspace{2mm}

\noindent 5.3.8   \hspace{4mm}       $\kappa_{\Lambda}(t) := \text{max} \{ \aF : \alpha \in \N^m \}$

\vspace{2mm}
\noindent is (well-defined and) definable. It also determines the topology on $\Lambda$ in the following sense.

\vspace{2mm}

\noindent {\bf Theorem 5.4}  \hspace{1mm} {\it Let $\Lambda$ be an $(R, m, K)$-family with $R>1$ as above and let $r$ be a real number satisfying $0<r<R$. Then there exists
a positive real number $B_{\Lambda} (r)$ such that for all $t \in T$ and all $z \in \Delta(r)^m$ we have
$|F_t(z)| \leq B_{\Lambda}(r) \cdot \kappa_{\Lambda} (t)$.}
\begin{proof}
Choose a real number $r_0$ such that $\text{max} \{ 1,r \} <r_0<R$ and for each $t \in T$ define $G_t : \Delta(\frac{R}{r_0})^m \to \C$
by $G_t(z) := F_t(r_0z)$. Then $\Lambda^{*} := \{G_t : t \in T \}$ is an $(\frac{R}{r_0} , m, K)$-family and since $\frac{R}{r_0} > 1$, we
may apply 5.3 to it and obtain some $M(\Lambda^{*}) \in \N$ such that
$\kappa_{\Lambda^{*}}(t) = \text{max} \{ \frac{|G_t^{(\alpha)}(0)|}{\alpha !} : \alpha \in \N^m , \hspace{1mm} |\alpha| \leq  M(\Lambda^{*} \}.$
Fix $t \in T$. Then since $G_t^{(\alpha)}(0) = r_0^{|\alpha|}F_t^{(\alpha)}(0)$ (for all $\alpha \in \N^m$) it follows that
for all $z \in \Delta(r)^m$ we have that $|F_t(z)| = |\sum_{\alpha \in \N^m} r_0^{-|\alpha|} \cdot \frac{G_t^{(\alpha)}(0)}{\alpha !} \cdot z^{\alpha}| \leq$
$\kappa_{\Lambda^{*}} (t) \cdot (\frac{r_0}{r_0 - r})^m \leq \kappa_{\Lambda} (t) \cdot r_0^{M(\Lambda^{*})} \cdot (\frac{r_0}{r_0 - r})^m$, which
gives the required result upon setting $B_{\Lambda}(r) :=  r_0^{M(\Lambda^{*})} \cdot (\frac{r_0}{r_0 - r})^m$.

\end{proof}

The topology we are referring to here is determined by the metrics $\delta_r$ ($0<r<R$) where, for any two bounded holomorphic
functions $F, G : \Delta(R)^m \to \C$, we define $\delta_r (F, G) := \text{sup} \{|F(z) - G(z)| : z \in \Delta (r)^m \}$. It turns out
that if $\Lambda$ is any $(R, m, K)$-family (regarded here as a set, rather than an indexed set, of functions)
and if $0< r, r' <R$, then the metric spaces $(\Lambda, \delta_r )$ and $(\Lambda, \delta_{r'} )$
are quasi-isometric via the identity function on $\Lambda$. (Note that this is certainly not true in general for families of
$K$-bounded holomorphic functions on $\Delta(R)^m$, e.g. consider, for $R=2, m=K=1$, the family $\{0 \} \cup \{ (\frac{z}{2})^q : q \in \N \}$.) In fact, as we now explain, they
are quasi-isometric to a bounded subset of $\C^N$ for some sufficiently large $N$ (depending only on $\Lambda$) endowed with
the metric induced by the usual sup-metric on $\C^N$: $\lVert \langle w_1 , \ldots , w_N \rangle \rVert := \text{max}\{ |w_i| : 1 \leq i \leq N \}$.

\vspace{2mm}

\noindent {\bf Definition 5.5}  \hspace{1mm} We say that an $(R, m, K)$-family $\Lambda = \{F_t : t \in T \}$ is {\it well-indexed} if,
for some $N \in \N$, $T$ is a bounded subset of $\C^N$ and
for each $r$ with $0<r<R$, there exist
positive real numbers $c_r$, $C_r$ such that for all $t, t' \in T$ we have $c_r \delta_r (F_t , F_{t'} ) \leq \lVert t-t' \rVert \leq   C_r \delta_r (F_t , F_{t'} )$.
That is, the map $t \mapsto F_t$ is a quasi-isometry from the metric space $\langle T, \lVert \cdot \rVert \rangle$ to the metric space
$\langle \Lambda , \delta_r \rangle$.

\vspace{3mm}

\noindent {\bf Theorem 5.6}  \hspace{1mm} {\it Let $\Lambda = \{F_t : t \in T \}$ be an $(R, m, K)$-family with $R>1$. Then there exists a well-indexed
$(R, m, K)$-family  $\Lambda' = \{G_t : t \in T^{*} \}$ such that $\Lambda = \Lambda'$ as sets. Further, $\text{dim}(T^{*}) \leq \text{dim}(T)$}.
\begin{proof}
Consider the $(R, m, 2K)$-family $\Omega := \{ F_t - F_{t'}: \langle t, t' \rangle \in T^2 \}$ and let $M(\Omega) \in \N$
be as given by 5.3 (with $\Omega$ in place of $\Lambda$). We take our $N = N(\Lambda)$ to be the cardinality of the set $\{ \alpha \in \N^m : |\alpha| \leq M(\Omega) \}$.

Define the map $\omega : T \to \C^N$ by $\omega(t) := \langle \frac{F_t^{(\alpha)} (0)}{\alpha !} : \alpha \in \N^m, |\alpha| \leq$

\vspace{1mm}

\noindent $M(\Omega) \rangle$ and set $T^{*} := \omega[T]$ (so obviously $\text{dim}(T^{*}) \leq \text{dim}(T)$) .  If $t, t' \in T$ and $\omega (t) =
\omega (t')$ then $\frac{F_t^{(\alpha)} (0)}{\alpha !} =  \frac{F_{t'}^{(\alpha)} (0)}{\alpha !}$ holds for all $\alpha$ with $|\alpha | \leq M(\Omega)$ and hence,
by 5.3 (and the linearity of the derivatives), it holds for  all $\alpha \in \N^m$. Thus $F_t = F_{t'}$.  This  does not necessarily imply
that $t=t'$ but, by the
principle of definable choice, we may  choose a definable right inverse $\omega^{-1}:T^{*} \to T$ of $\omega$ and, setting  $G_t = F_{\omega^{-1}(t)}$ (for $t \in T^{*}$),
we have that, as a set,  $\Lambda =  \{ G_{t} : t \in T^{*} \}$. We complete the proof by showing that the $(R,m,K)$-family $\{ G_{t} : t \in T^{*} \}$ is  well-indexed.

Firstly,  by 5.1.1 we have that $T^{*} \subseteq \overline{\Delta(K)}^N$, so $T^{*}$ is a bounded subset of $\C^N$. For the quasi-isometric inequalities
consider some  $r \in (0, R)$.

We may take $C_r = \text{max} \{ 1, r^{-M(\Omega)} \}$. Indeed, suppose $t, t' \in T^{*}$. Let $s = \omega^{-1}(t)$
and $s' = \omega^{-1}(t')$. Then by the Cauchy inequalities
applied to the function $(F_{s} - F_{s'})$ restricted to the disk $\Delta(r)^m$, we have that for all $\alpha \in \N^m$, $\frac{|(F_{s} - F_{s'})^{\alpha}(0)|}{\alpha !}
\leq  \frac{\delta_r(F_s , F_{s'})}{r^{|\alpha|}}$. In particular, $\lVert t-t' \rVert = \lVert \omega(s) - \omega(s') \rVert \leq C_r \delta_r (F_s , F_{s'}) =
C_r \delta_r (G_ {t}, G_{t'})$.

Finally, we take $c_r$ to be $B_{\Omega}(r)^{-1}$, where $B_{\Omega}(r)$ is as in 5.4 (with $\Omega$ in place of $\Lambda$).
Then, with $t,t',s,s'$ as above, and $z \in \Delta(r)^m$ we have by 5.3 and 5.4, $|(F_s - F_{s'})(z)| \leq B_{\Omega}(r)
\cdot \kappa_{\Omega}(\langle s, s' \rangle) = B_{\Omega}(r) \cdot \lVert \omega(s) - \omega(s') \rVert =
B_{\Omega}(r) \cdot \lVert t - t' \rVert$. So  $\lVert t - t' \rVert \geq c_r \delta_r (G_ {t}, G_{t'})$, as required.

\end{proof}

This result suggests a natural way of compactifying definable $(R, m, K)$-families. For let $\Lambda = \{ F_t : t \in T \}$
be such a family with $R>1$ and assume, as now we may, that it is well-indexed (with $T$  a bounded subset of $\C^N$, say). We wish
to extend $\Lambda$ to a family $\overline{\Lambda}$ well-indexed by the closure $\overline{T}$ of $T$ in $\C^N$. So for $t \in \overline{T}$
choose a Cauchy sequence $\langle t^{(i)} : i \in \N \rangle$ in $T$ converging to $t$   (in the space $\langle \C^N , \lVert \cdot \lVert \rangle$).
Then by the quasi-isometric property of the indexing it follows that $\langle F_{t^{(i)}} : i \in \N \rangle$ is a Cauchy sequence in
$\langle \Lambda , \delta_r \rangle$ for every $r \in (0, R)$. So by Weierstrass' theorem on uniformly convergent sequences, there exists a holomorphic function
$F_t : \Delta(R)^m \to \C$ such that, for each $r \in (0, R)$, $\delta_r (F_{t^{(i)}}, F_t) \to 0$ as $i \to \infty$. It is easy to
check that $F_t$ depends only on $t$ (and not on the particular choice of Cauchy sequence) and that our notation is consistent if $t$
happens to lie in $T$. We have the following
\vspace{2mm}

\noindent {\bf Theorem 5.7}  \hspace{1mm} {\it The collection $\overline{\Lambda} := \{ F_t : t \in \overline{T} \}$
 as defined above is a well-indexed $(R, m, K)$-family.}

\begin{proof}
 Everything follows from elementary facts on convergence (and we may take the same constants $c_r$, $C_r$ for the quasi-isometric
 inequalities) apart from the definability of $\overline{\Lambda}$. To see that this holds too, let
 $$\text{graph}(\Lambda) := \{ \langle t, z, w \rangle \in \C^{N+m+1} : t \in T, z \in \Delta(R)^m , F_t (z) = w \}.$$
 Then $\text{graph}(\Lambda)$ is a definable subset of $\C^{N+m+1}$ (this being the definition of what it means for $\Lambda$
 to be a definable family). We complete the proof by showing that

 \vspace{2mm}

 \noindent(*) for all $\langle t, z, w \rangle \in \C^{N+m+1}$, \hspace{2mm}  $t \in \overline{T}, z \in \Delta(R)^m$ and $F_t (z) = w$
 if and only if $z \in \Delta(R)^m$ and $\langle t, z, w \rangle \in \overline{\text{graph}(\Lambda)}$.

 \vspace{2mm}

 So let $\langle t, z, w \rangle \in \C^{N+m+1}$.

 Suppose first that $t \in \overline{T}, z \in \Delta(R)^m$ and $F_t (z) = w$.
Choose a sequence $\langle t^{(i)} : i \in \N \rangle$ in $T$ converging to $t$. Choose $r$ so that $|z|<r<R$. Then by the construction
of $F_t$ we have that $\delta_r (F_{t^{(i)}}, F_t) \to 0$ as $i \to \infty$. In particular, $|F_{t^{(i)}}(z) - F_t(z)| \to 0$ as $i \to \infty$,
i.e. $F_{t^{(i)}}(z) \to w$ as $i \to \infty$. So $\langle t^{(i)} , z , F_{t^{(i)}}(z) \rangle \to \langle t, z, w \rangle$ as $i \to \infty$.
Since   $\langle t^{(i)} , z , F_{t^{(i)}}(z) \rangle \in \text{graph}(\Lambda)$ for each $i \in \N$, it follows
that $\langle t, z, w \rangle \in \overline{\text{graph}(\Lambda)}$ as required.

For the converse, suppose that $z \in \Delta(R)$ and that $\langle t, z, w \rangle \in \overline{\text{graph}(\Lambda)}$. Then
certainly $t \in \overline{T}$ and we must show that $F_t(z) = w$, thereby completing the proof of (*).

Let $\langle \langle t^{(i)}, z^{(i)} , w_i \rangle : i \in \N \rangle$ be a sequence in $\text{graph}(\Lambda)$ converging to
$\langle t, z, w \rangle$. Then $z^{(i)} \to z$ as $i \to \infty$ and since $z \in \Delta(R)$, we may choose $r<R$ so that
 $z \in \Delta(r)^m$ and $z^{(i)} \in \Delta(r)^m$ for each $i \in \N$. Since $t^{(i)} \to t$ as $i \to \infty$, it follows from
 the construction of $F_t$ that $\delta_r (F_{t^{(i)}}, F_t) \to 0$ as $i \to \infty$. In particular, $|F_{t^{(i)}}(z^{(i)}) - F_t(z^{(i)})| \to 0$ as $i \to \infty$.
But by the definition of $\text{graph}(\Lambda)$, $F_{t^{(i)}}(z^{(i)}) = w_i$ for all $i \in \N$ and hence $|w_i - F_t(z^{(i)})| \to 0$ as $i \to \infty$.
However, $F_t(z^{(i)}) \to F_t(z)$ as $i \to \infty$ (because $F_t$ is certainly continuous on $\Delta(r)^m$) and hence $w_i \to F_t(z)$ as $i \to \infty$.
Since $w_i \to w$ as $i \to \infty$ it now follows that $w = F_t(z)$ as required.
\end{proof}

Having shown how to compactify $(R,m,K)$-families, we now projectivize them.

\vspace{2mm}

\noindent {\bf Theorem 5.8}  \hspace{1mm} {\it Let $\Lambda = \{ F_t : t \in T \}$ be an $(R,m,K)$-family with $R>1$. Assume that for no $t \in T$ does
$F_t$ vanish identically. Let $R_0$ satisfy $1<R_0 < R$. Then there exists a positive real number $K_0$ and an  $(R_0, m, K_0 )$-family
$\Lambda^{\dag} = \{ G_t : t\in T^{\dag} \}$ such that

\vspace{1mm}

\noindent5.8.1 \hspace{2mm} $\Lambda^{\dag}$ is well-indexed  and $T^{\dag}$ is closed in its ambient space $\C^N$;

\vspace{1mm}

\noindent5.8.2 \hspace{2mm} for every $t \in T$, there exists $A_t > 0$ and $t^{\dag} \in T^{\dag}$ such that \newline
$G_{t^{\dag}} = A_t \cdot F_t \upharpoonright \Delta (R_0 )^m$;

\vspace{1mm}

\noindent5.8.3 \hspace{2mm} the (real) dimension of $T^{\dag}$ is at most that of $T$;

\vspace{1mm}

\noindent5.8.4 \hspace{2mm} for no $t \in T^{\dag}$ is $G_t$ identically zero.}

\begin{proof}
We consider the $(R_0 , m , K_0)$-family $\{ \frac{F_t}{\kappa_{\Lambda}(t)} \upharpoonright \Delta (R_0)^m : t \in T \}$
(cf. 5.3.8), where $K_0 = B_{\Lambda} (R_0)$ (cf. 5.4). Using 5.6, let $\Lambda^{*} = \{ G_t : t \in T^{*} \}$ be a well-indexing
of it. Then $\text{dim}(T^{*}) \leq \text{dim}(T)$. We set $T^{\dag} := \overline{T^{*}}$ and $\Lambda^{\dag} := \overline{\Lambda^{*}}$
as in 5.7. Then 5.8.1-3 are clear. For 5.8.4, let us first note that if $t \in T^{*}$ then for some $s \in T$, $G_t = \frac{F_s}{\kappa_{\Lambda}(s)} \upharpoonright \Delta (R_0)^m$
and hence there exists $\alpha \in \N^m$ with $|\alpha| \leq M(\Lambda)$  such that $\frac{|G_t^{(\alpha)}(0)|}{\alpha !} = 1$ (see 5.3 and 5.3.8).
Now let $t^{\dag} \in \overline{T^{*}}$. We must show that $G_{t^{\dag}}$ does not vanish identically. For this, choose $t \in T^{*}$ such that $\rVert t-t^{\dag}\lVert < \frac{c_1}{2}$
so that $\delta_1 (G_t ,G_{t^{\dag}} ) < \frac{1}{2}$ (by 5.5 with $r=1$). It now follows from the Cauchy inequalities applied to the function
$G_t - G_{t^{\dag}}$ restricted to the unit polydisk $\Delta(1)^m$, that for all $\alpha \in \N^m$ we have $\frac{|G_t^{(\alpha)}(0)-G_{t^{\dag}}^{(\alpha)}(0)|}{\alpha !} < \frac{1}{2}$.
So choosing $\alpha$ with $\frac{|G_t^{(\alpha)}(0)|}{\alpha !} = 1$ as above, we see that $\frac{|G_{t^{\dag}}^{(\alpha)}(0)|}{\alpha !} > \frac{1}{2}$. In particular,
$G_{t^{\dag}}$ does not vanish identically.

\end{proof}

\noindent {\bf Remark}  \hspace{1mm}
The hypothesis that $R_0 < R$ is necessary here:  the reader may easily verify that for
the $(2,1,1)$-family $\Lambda = \{ g_t : t \in [0, \frac{1}{2}) \}$ where $g_t(z) = \frac{1-2t}{1-tz}$ , and for each
given $K_0 > 0$, there
is no $(2, 1, K_0)$-family $\Lambda^{\dag}$ satisfying 5.8.1-4.

\vspace{3mm}

We are almost ready for the proof of the quasi-parameterization theorem (2.2.3). This will proceed by induction on the dimension of the
given family $\{ X_t : t \in T \}$, i.e. the (minimum, real) dimension of the indexing set $T$. The inductive step will involve
a use of the Weierstrass Preparation Theorem (or, rather, a modification of the argument used in the complex analytic proof of the Weierstrass
Preparation Theorem) and, as usual, one first has to make a transformation so that the function being prepared is regular in one of its
variables. Further, in our case the transformation will have to work uniformly for all members of a certain definable family of functions and for all values of the other variables.
Unfortunately, the usual linear change of variables does not have this property. Instead we use a variation of the transformation
used by Denef and van den Dries in their proof of quantifier elimination for the structure $\R_{an}$ (see \cite{DvdD}). The result we require
is contained in the following

\vspace{2mm}

\noindent {\bf Theorem 5.9}  \hspace{1mm} {\it Let $\Lambda = \{ F_t : t \in T \}$ be an $(R,m,K)$-family with  $R>1$ (and, for non-triviality, with $m \geq 2$)
such that for no $t \in T$ does $F_t$ vanish identically. Let $R'$ and $R''$ be real numbers
satisfying $1<R''<R'<R$. Then there exist positive integers $D_1,\ldots,D_{m-1}$ and a positive real number $\eta$ such that the
bijection} $${\it \theta : \C^{m} \to \C^{m} : z = \langle z_1,\ldots,z_m \rangle \mapsto \langle z_1 + \eta z_m^{D_1},\ldots,z_{m-1} + \eta z_m^{D_{m-1}}, z_m \rangle}$$
{\it satisfies

\vspace{2mm}

\noindent 5.9.1 \hspace{2mm} $\theta [ \overline{\Delta(R')}^m ] \subseteq \Delta(R)^m$,

\vspace{2mm}

\noindent 5.9.2 \hspace{2mm} $\theta^{-1}[\overline{\Delta(1)}^m] \subseteq \Delta(R'')^m$, and

\vspace{2mm}

\noindent 5.9.3 \hspace{2mm} for each $t \in T$ and $z' \in \Delta(R')^{m-1}$ the function $z_m \mapsto F_t \circ \theta (z',z_m)$ (for
$z_m \in \Delta(R')$) does not vanish identically in $z_m$.}

\vspace{3mm}
This will follow from the following general

\vspace{2mm}

\noindent {\bf Lemma 5.10} \hspace{1mm} {\it Let $m \geq 1$ and suppose that $\mathcal{X} = \{ X_t : t \in T \}$ is a definable family of subsets of $\R^m$ such that
for all $t \in T$, $\text{dim}(X_t) < m$. Then there exist positive integers $D_1,\ldots,D_{m-1}$ such that for all
$t \in T$, all $\eta > 0$ and all $w_1,\ldots, w_{m-1} \in \R$, there exists $\epsilon = \epsilon (t, \eta, w_1,\ldots, w_{m-1} ) >0$
such that} $${\it X_t \cap \{ \langle w_1 + \eta x^{D_1},\ldots,w_{m-1} + \eta x^{D_{m-1}}, x \rangle \in \R^m : 0<x< \epsilon \} = \emptyset.}$$
\begin{proof}
Induction on $m$. For $m=1$, each $X_t$ is (uniformly) finite. So obviously we can find, for each $t \in T$, an $\epsilon = \epsilon (t) > 0$
such that $X_t \cap (0, \epsilon) = \emptyset$, which is the required conclusion in this case.

Now assume that the lemma holds for some $m \geq 1$ and that $\{ X_t : t \in T \}$ is a definable family of subsets of $\R^{m+1}$ each having
dimension at most $m$.

For $t \in T$ define $S_t := \{ s \in \R^m : \{ y \in \R : \langle y, s \rangle \in X_t \} \hspace{1mm} \text{is infinite} \}$. Then $\{ S_t : t \in T \}$
is a definable family of subsets of $\R^m$ and clearly $\text{dim}(S_t )<m$ for each $t \in T$. So we may apply the inductive hypothesis
to this family and obtain (with a small shift in notation) positive integers $D_2 , \ldots,D_m$ such that for all $t \in T$, all
$\eta >0$ and all $w_2,\ldots,w_m \in \R$, there exists $\epsilon = \epsilon (t, \eta, w_2,\ldots, w_{m} ) >0$ such that for all
$x \in (0, \epsilon)$, we have that $\langle w_2 + \eta x^{D_2},\ldots,w_{m} + \eta x^{D_{m}}, x \rangle \notin S_t$, i.e. there
are at most finitely many $y \in \R$ such that $\langle y, w_2 + \eta x^{D_2},\ldots,w_{m} + \eta x^{D_{m}}, x \rangle \in X_t$.

Now, by the principle of definable choice, there exists a definable function $H: T \times (0, \infty) \times \R^m \times \R \to (0, 1]$
such that for all $t \in T$, all $\eta \in \R$, all $w \in \R^m$ and all $x \in \R$, its value $H(t, \eta , w, x)$ is
some $y \in (0,1)$ such that for no $u \in (0,y)$ do we have
 $\langle w_1 + \eta u, w_{2} + \eta x^{D_{2}},\ldots, w_{m} + \eta x^{D_{m}}, x \rangle \in X_t$, if such a $y$ exists (and is, say, $1$ otherwise).
Notice that by the discussion above, such a $y$ does indeed exist whenever $x \in (0, \epsilon (t, \eta, w_2,\ldots,w_m))$.

We now apply polynomial boundedness to obtain a positive integer $D_1$ such that for all $t \in T$, all $\eta > 0$ and all $w = \langle w_1 ,\ldots, w_m \rangle \in \R^m$, there
exists $\gamma = \gamma (t, \eta , w) > 0$, which we may assume is strictly less than  $\epsilon (t, \eta, w_2,\ldots,w_m)$,
such that for all $x \in (0 , \gamma )$, we have $H(t, \eta , w, x) > x^{D_1}$. So if $x \in (0, \gamma )$, then $x \in (0, \epsilon )$ and hence
for all $u \in (0, H(t, \eta , w, x))$ we have that $\langle w_1 + \eta u, w_{2} + \eta x^{D_{2}}, \ldots, w_{m} + \eta x^{D_{m}}, x \rangle
\notin X_t$. But $x^{D_1} \in (0, H(t, \eta , w, x) )$ so $\langle w_1 + \eta x^{D_1}, w_{2} + \eta x^{D_{2}}, \ldots, w_{m} + \eta x^{D_{m}}, x \rangle
\notin X_t$. Since this holds for arbitrary $x \in (0, \gamma )$, we are done (upon taking  $\epsilon (t, \eta, w_1,\ldots,w_m) := \gamma (t, \eta, w)$).

\end{proof}

\noindent
{\it Proof of 5.9}

For $t \in T$ and $u \in (-R, R)^m$, let
$X_{\langle t, u \rangle} := \{ w \in (-R', R')^m : w+ u\text{i} \in \Delta (R)^m \hspace{2mm}\text{and} \hspace{2mm} F_t(w+  u\text{i}) = 0 \}$.
Then $\text{dim}(X_{\langle t, u \rangle}) < m$ because if  $U$ is some non-empty, open subset of $(-R', R')^m$ such that
$w+ u\text{i} \in \Delta (R)^m \hspace{2mm}\text{and} \hspace{2mm} F_t(w+  u\text{i}) = 0$ for all $w \in U$,
then $F_t$, being holomorphic, would vanish
identically on $\Delta(R)^m$ which is contrary to hypothesis. So we may apply 5.10 to the family \newline
$\mathcal{X} := \{ X_{\langle t, u \rangle} : \langle t, u \rangle \in T \times (-R, R)^m \}$ and obtain
positive integers $D_1, \ldots,D_{m-1}$ with the property stated in the conclusion of 5.10. Now choose $\eta$ so small
that the resulting map $\theta$ satisfies 5.9.1 and 5.9.2. (The inverse of $\theta$ is given by
$\theta^{-1} (z_1, \ldots,z_m) = \langle z_1 - \eta z_m^{D_1},  \ldots, z_{m-1} - \eta z_{m}^{D_{m-1}}, z_m \rangle$.)

To verify 5.9.3, let $t \in T$ and let $z' = \langle z_1,\ldots,z_{m-1} \rangle \in \Delta (R')^{m-1}$.
 If $F_t \circ \theta (z' , z_m ) = 0$ for all
$z_m \in \Delta(R')$ then, in particular, $\langle z_1 + \eta x^{D_1},\ldots,z_{m-1} + \eta x^{D_{m-1}}, x \rangle \in
\Delta(R')^m$
 and $F_t (z_1 + \eta x^{D_1},\ldots,z_{m-1} + \eta x^{D_{m-1}}, x ) = 0$
for all sufficiently small positive $x \in \R$. However, if the real and imaginary parts of $z_i$ are, respectively,
$a_i$ and $b_i$ (for $i = 1,\ldots,m-1$),   this implies that
$\langle a_1 + \eta x^{D_1},\ldots,a_{m-1} + \eta x^{D_{m-1}}, x \rangle \in X_{\langle t, b \rangle}$ for all sufficiently
small $x>0$, where $b:= \langle b_1,\ldots,b_{m-1} , 0 \rangle$. But this clearly contradicts the conclusion of 5.10.

\vspace{5mm}

We now come to the proof of the quasi-parameterization theorem (2.2.3), so definability is now, and henceforth, with
respect to a structure as described in 2.2.1.

Recall that we are given a definable family
$\mathcal{X} = \{ X_s : s \in S \}$ of subsets of $[-1, 1 ]^n$ each of dimension at most $m$, where $m<n$.
We assume that the indexing set $S$ has been chosen of minimal dimension and we denote this dimension by $\text{indim}(\mathcal{X})$.
We are required to find some $R>1, K>0$, a positive integer $d$, and an
$(R, m+1, K)$-family $\Lambda^{*}$, each element of which is a monic polynomial of degree at most $d$ in its first variable,
such that

\vspace{2mm}

\noindent
(*) \hspace{2mm} for all $s \in S$, there exists $F \in \Lambda^{*}$ such that
$X_s \subseteq \{ x = \langle x_1,\ldots,x_n \rangle \in [-1, 1]^n : \exists w \in [-1, 1]^m \bigwedge_{i=1}^n F(x_i ,w ) =0 \}.$

\vspace{4mm}

Let us first consider the case $\text{indim}(\mathcal{X})=0$, i.e. the case that $\mathcal{X}$ is finite. In fact, it is sufficient to consider
the case that
$\mathcal{X}$ consists of a single set, $X$ say, where $X \subseteq [-1,1]^n$ and $\text{dim}(X) \leq m<n$.

\vspace{1mm}

\noindent
{\bf Remark 5.11} \hspace{1mm} Indeed, it is obvious that, in general, if the conclusion of the quasi-parameterization theorem holds for the families $\{ X_s : s \in S_1 \}$
and $\{ X_s : s \in S_2 \}$, then it also holds for the family $\{ X_s : s \in S_1 \cup S_2 \}$.

\vspace{1mm}

Since we are now assuming that our ambient
o-minimal structure is a reduct of $\R_{an}$, we may apply the $0$-mild parameterization theorem (Proposition 1.5 of \cite{JMT})
which tells us that (after routine translation and scaling) there exists a finite set $\{ \Phi_j : 1 \leq j \leq l \}$ of
definable, real analytic maps $\Phi_j = \langle \phi_{j,1},\ldots,\phi_{j,n} \rangle :(-3,3)^m \to \R^n$ (say) whose images
on $[-1,1]^m$ cover $X$, and are such that $\frac{|\phi^{(\alpha)}_{j,i} (0)|}{\alpha !} \leq c\cdot3^{-|\alpha|}$ for some constant $c$, and all
$\alpha \in \N^m$, $j = 1,\ldots,l$ and $i=1,\ldots,n$. If we now invoke our other assumption on the ambient 0-minimal structure,
then (again, after translation and scaling at the expense of increasing the number of parameterizing functions)
there is no harm in assuming that each function $\phi_{j.i}$ has a definable, complex extension
(for which we use the same notation) to the polydisk $\Delta(2)^m$.

We now set
$$F(x, w) := \prod_{j=1}^{l}\prod_{i=1}^{n}(x - \phi_{j,i}(w))$$
for $x \in \Delta(2)$, $w \in \Delta(2)^m$.

\vspace{3mm}

Then $F$ is a monic polynomial of degree $d=ln$ in its first variable. Further, $\{ F \}$ is, for some $K>0$, a $(2,m+1,K)$-family which
clearly has the required property (*).

\vspace{3mm}

We now proceed by induction on $\text{indim}(\mathcal{X})$. So consider some $k,m,n \in \N$ with $k = \text{indim}(\mathcal{X})\geq 1$ and $m<n$, and a definable family
$\mathcal{X} = \{X_s : s \in S \}$ of subsets of $[-1,1]^n$ with $\text{dim}(X_s) \leq m$ for each $s \in S$, and assume
that the theorem holds for families of indim $< k$ (for arbitrary $m,n$). Now it is easy to show that  we may represent $\mathcal{X}$ in the form
$\{ X_u : u \in [-1, 1 ]^k \}$.

In order to apply the inductive hypothesis we define the family $\mathcal{Y} := \{ Y_{u'} : u' \in [-1,1]^{k-1} \}$ of subsets of $[-1,1]^{n+1}$ where,
for each $u' \in [-1,1]^{k-1}$,

\vspace{2mm}

\noindent
5.12 \hspace{3mm} $Y_{u'} := \{ \langle x, u_{k} \rangle \in [-1,1]^{n+1} : x \in X_u \}.$

\vspace{2mm}

(In the course of this proof we shall use the convention that if $v$ is a tuple whose length, $p$ say, is clear from the context, then
$v = \langle v_1,\ldots,v_p \rangle$, and $v' = \langle v_1,\ldots,v_{p-1} \rangle$. Also, by convention, $[-1, 1]^0 := \{0\}$.)

Clearly $\text{indim}(\mathcal{Y}) < k$ and, for each $u' \in  [-1,1]^{k-1}$, $\text{dim}(Y_{u'}) \leq m+1 < n+1$, so we may indeed apply
our inductive hypothesis to $\mathcal{Y}$ and obtain some $R>1$, $K>0$,  an $(R,m+2,K)$-family $\Lambda = \{ H_t : t \in T \}$, and
a positive integer $d$ such that

\vspace{2mm}

\noindent
5.13 \hspace{2mm} each $H_t$ is a monic polynomial of degree at most $d$ in its first variable, and

\vspace{2mm}

\noindent
5.14 \hspace{2mm} for each $u' \in [-1,1]^{k-1}$ there exists $t = t(u') \in T$ such that \newline
$Y_{u'} \subseteq \{ \langle x, x_{n+1} \rangle  \in [-1, 1]^{n+1} :
\exists w \in [-1, 1]^{m+1} (\bigwedge_{i=1}^{n+1} H_t (x_i ,w ) =0 \}.$

\vspace{2mm}

In order to prepare the functions in $\Lambda$ as discussed above, we must first remove those $t$ from $T$
such that for some $z_1$ the function $H_t(z_1, \cdot)$ vanishes identically (in its last $m+1$ variables).
To do this, we first note that, by the principle of definable choice, the correspondence $u' \mapsto t(u')$
(for $u' \in [-1,1]^{k-1}$) may be taken to be a definable function and so the set \newline $E := \{ u \in [-1, 1]^{k}
: H_{t(u')} (u_{k}, {\bf 0}) = 0 \}$ is definable (where ${\bf0}$ is the origin of $\R^{m+1}$). We have
$\text{dim}(E) < k$ because if $E$ contained a non-empty open subset of $[-1,1]^{k}$, then we could find
some $u' \in [-1,1]^{k-1}$ such that $H_{t(u')} (u_{k}, {\bf 0}) = 0$ for all $u_{k}$ lying in some non-empty
open interval, which is impossible as $H_{t(u')} (\cdot, {\bf 0})$ is a monic polynomial.

Thus, by another use of the inductive hypothesis, the family $\{ X_u : u \in E \}$ satisfies the conclusion
of the quasi-parameterization theorem and so, by 5.11, it is sufficient to consider the family
$\{ X_u : u \in [-1,1]^k \setminus E \}$.

For this we define, for each $u \in [-1,1]^k \setminus E$, the function $H^{*}_u : \Delta (R)^{m+1} \to \C$
by

\vspace{2mm}

\noindent
5.15 \hspace{2mm} $H^{*}_u (z) := H_{t(u')} (u_k , z)$,

\vspace{2mm}

\noindent
so that for all $u \in [-1,1]^k \setminus E $, $H^{*}_u({\bf 0}) \neq 0$. Now set

\vspace{2mm}

\noindent
5.16 \hspace{2mm} $\Lambda_0 := \{ H_u^{*} :u \in [-1, 1]^k \setminus E \}$.

\vspace{2mm}

Then $\Lambda_0$ is an $(R, m+1, K)$-family which does not contain the zero function. So we may apply 5.8 to it
with, say, $R_0 = \frac{1+R}{2}$ (and $m+1$ in place of $m$) and obtain, for some $K_0> 0$, an $(R_0, m+1, K_0)$-family
$\Lambda_0^{\dag} = \{ G_t : t \in T_0^{\dag} \}$  having properties 5.8.1-4.

I claim that

\vspace{2mm}

\noindent
5.17 \hspace{2mm} for all $u \in [-1,1]^k \setminus E$, there exists $t^{\dag} \in T_0^{\dag}$ such that

\vspace{1mm}

\noindent
$X_u \subseteq \{ x \in [-1,1]^n : \exists w \in [-1,1]^{m+1} ( \bigwedge_{i=1}^n H_{t(u')} (x_i , w)=0 \wedge G_{t^{\dag}}(w)=0) \}$.

\vspace{2mm}

Indeed, let $u \in [-1,1]^k \setminus E$.  By 5.8.2, there is
some $t^{\dag} \in T_0^{\dag}$ and $A>0$ such that for all $z \in \Delta(R_0)^{m+1}$,

\vspace{2mm}

\noindent
5.17.1 \hspace{2mm} $G_{t^{\dag}}(z) = A \cdot H^{*}_u (z)$.

\vspace{2mm}

Now let $x \in X_u$. Then, by 5.12, $\langle x, u_k \rangle \in Y_{u'}$. Hence, by 5.14, we may choose
$w \in [-1,1]^{m+1}$ such that $H_{t(u')}(x_i ,w) =0$ for $i = 1, \ldots, n$ and $H_{t(u')}(u_k ,w) = 0$.
Since $[-1,1]^{m+1} \subseteq \Delta(R_0)^{m+1}$, 5.17 now follows from 5.17.1 and 5.15.

\vspace{3mm}

In order to complete the proof we must reduce the range of the $w$-variable in 5.17 from $[-1,1]^{m+1}$ to
$[-1,1]^{m}$. The idea is simple: we use the relation $G_{t^{\dag}}(w)=0$ to express $w_{m+1}$ as a function of  $w_1, \ldots,w_m$,
and then substitute this function for $w_{m+1}$ in the first conjunct appearing in 5.17. Of course, there are some technical difficulties to
be overcome. Firstly, we must ensure that $G_{t^{\dag}}(w)$ really does depend on $w_{m+1}$ and this is achieved
by the transformation described in 5.9. Secondly, the argument only works locally. However, the compactness of $T_0^{\dag}$ will guarantee
that this is sufficient. And finally, the functional dependence of $w_{m+1}$ on $w_1,\ldots,w_m$ will, in general, be a many-valued one.
This is precisely why we only obtain {\it quasi}-parameterization rather than parameterization.

\vspace{1mm}

So, to carry out the first step, we apply 5.9 to the $(R_0,m+1,K_0)$-family $\Lambda_0^{\dag} = \{ G_t : t \in T_0^{\dag} \}$ (which
is permissible as it satisfies 5.8.4) with $R' = 1+\frac{2(R_0 -1)}{3}$ and $R'' = 1+\frac{(R_0 -1)}{3}$ (and $m+1$ in place of $m$).
Let $\theta : \C^{m+1} \to \C^{m+1}$ be as in 5.9 and, for each $t \in T_0^{\dag}$ set

\vspace{2mm}

\noindent
5.18 \hspace{2mm} $\tilde{G_t} := G_t \circ \theta \upharpoonright \Delta (R')^{m+1}$.

\vspace{2mm}

Then $\{ \tilde{G_t} : t \in T_0^{\dag} \}$ is an $(R',m+1, K_0)$-family. Further, since the family $\Lambda_0^{\dag}$
is well indexed (5.8.1), it immediately follows (from 5.5 and the Cauchy inequalities) that for each $\alpha \in \N^{m+1}$,
the function $G^{(\alpha)}_t(z)$ is continuous in both $t$ and $z$,
for $\langle t, z \rangle \in T_0^{\dag} \times \Delta(R_0)^{m+1}$ . Since $\theta$ is holomorphic throughout $\C^{m+1}$
we obtain

\vspace{2mm}

\noindent
5.19 \hspace{2mm} for each $\alpha \in \N^{m+1}$, the function $\tilde{G}^{(\alpha)}_t(z)$ is continuous in both $t$ and $z$ for
$\langle t, z \rangle \in T_0^{\dag} \times \Delta(R')^{m+1}$.

\vspace{2mm}

Also, it follows from 5.9.3 that

\vspace{2mm}

\noindent
5.20 \hspace{2mm} for all $t \in T_0^{\dag}$ and all $z' \in \Delta(R')^m$, the function $\tilde{G_t} (z', \cdot)$
does not vanish identically on $\Delta (R')$.

\vspace{2mm}

Having modified the functions $G_t$, we must now adjust the functions $H_{t(u')}$ in order to preserve 5.17. Accordingly, we define,
for each $u \in [-1,1]^{k} \setminus E$, the function $\tilde{H}_{u} : \Delta(R')^{m+2} \to \C$ (which, in fact, only
depends on $u'$) by

\vspace{2mm}

\noindent
5.21 \hspace{2mm} $\tilde{H}_{u} (z_1,z_2,\ldots,z_{m+2}) := H_{t(u')} (z_1, \theta (z_2,\ldots,z_{m+2})).$

\vspace{2mm}

Then $\{ \tilde{H}_{u} : u \in [-1,1]^k \setminus E \}$ is an $(R', m+2, K)$-family and, as we show below,  the following
version of 5.17 holds.

\vspace{2mm}

\noindent
5.22 \hspace{2mm} for all $u \in [-1,1]^k \setminus E$, there exists $t^{\dag} \in T_0^{\dag}$ such that

\vspace{1mm}

\noindent
$X_u \subseteq \{ x \in [-1,1]^n : \exists v \in (-R'',R'')^{m+1} ( \bigwedge_{i=1}^n \tilde{H}_{u} (x_i , v)=0 \wedge \tilde{G}_{t^{\dag}}(v)=0) \}$.

\vspace{2mm}

Indeed, let $u = \langle u', u_{k} \rangle \in [-1,1]^k \setminus E$ and choose $t^{\dag} \in T_0^{\dag}$ as in 5.17. Suppose $x \in X_u$
and (by 5.17) choose $w \in [-1,1]^{m+1}$ such that  $\bigwedge_{i=1}^n H_{t(u')} (x_i , w)=0 \wedge G_{t^{\dag}}(w)=0$. Let $v = \theta^{-1}(w)$.
Then by 5.9.2, $v \in \Delta(R'')^{m+1}$. But all the coordinates of $v$ are real, so $v \in (-R'',R'')^{m+1}$. Also, for $i=1,\ldots,n$ we have,
by 5.21, that $\tilde{H}_{u} (x_i , v) = H_{t(u')} (x_i , \theta(v)) =  H_{t(u')} (x_i , w) = 0$. Similarly, by 5.18,
$\tilde{G}_{t^{\dag}} (v) = G_{t^{\dag}} (w) =0$ and 5.22 follows.

Let us also record here the fact that in view of 5.13, and since the transformation 5.21 does not affect the variable $z_1$, we have

\vspace{2mm}

\noindent
5.23 \hspace{2mm} for each $u \in [-1,1]^k \setminus E$, the function $\tilde{H}_u$ is a monic polynomial of degree at most $d$ in its first variable.

\vspace{3mm}

We now carry out the local argument, as sketched above, that expresses $z_{m+1}$ as a many-valued function of $z' = \langle z_1,\ldots,z_m \rangle$ via
the relation  $\tilde{G}_{t^{\dag}}(z', z_{m+1}) = 0$.

Firstly, fix some $R_1$ with $R'' < R_1 < R'$ and for each $r$ with $R''<r<R_1$ let $C_r$ be the circle in $\C$ with centre $0$ and
radius $r$. Consider the set

\vspace{2mm}

\noindent
5.24  \hspace{2mm} $V_r := \{ \langle t, z' \rangle \in T_0^{\dag} \times \overline{\Delta(R_1)}^m : \hspace{1mm} \text{for all} \hspace{2mm}
z_{m+1} \in C_r , \hspace{2mm}  \tilde{G}_t(z', z_{m+1}) \neq 0 \}$.

\vspace{2mm}

It follows from 5.19 that $V_r$ is an open subset of $T_0^{\dag} \times \overline{\Delta(R_1)}^m$ (for the $\lVert \cdot \rVert$-metric
inherited from $\C^{N+m}$, where $N$ is as in 5.8.1). Further, it follows easily from 5.20 that the collection $\{ V_r : R''<r<R_1 \}$
covers the compact space $T_0^{\dag} \times \overline{\Delta(R_1)}^m$.

Now, by the Lebesgue Covering Lemma, there exists $\epsilon > 0$, a positive integer $M$, and points $t^{(1)},\ldots,t^{(M)} \in T_0^{\dag}$,
$a^{(1)},\ldots,a^{(M)} \in [-R'', R'']^m$ such that

\vspace{2mm}

\noindent
5.25 \hspace{2mm} the collection $\{ t^{(h)} + \Delta(\epsilon)^N : h=1,\ldots, M \}$ covers $T_0^{\dag}$,

\vspace{2mm}

\noindent
5.26 \hspace{2mm} each set $a^{(j)} + \Delta(2\epsilon)^m$ is contained in $\Delta(R_1)^m$
and the collection $\{ a^{(j)} + (-\epsilon, \epsilon )^m : j=1,\ldots, M \}$ covers $[-R'',R'']^m$, and

\vspace{2mm}

\noindent
5.27 \hspace{2mm} for each $h, j = 1, \ldots , M$, there exists $r_{h,j} \in (R'', R_1)$ such that \newline
$(\langle t^{(h)}, a^{(j)} \rangle + \overline{\Delta(2 \epsilon)}^{N+m})  \cap (T_0^{\dag} \times \Delta(R_1)^m) \subseteq V_{r_{h,j}}.$

\vspace{3mm}

Fix, for the moment, $h, j \in \{ 1,\ldots,M \}$. Then for each \newline $t \in  T_0^{\dag} \cap (t^{(h)} + \Delta(2\epsilon)^N)$ and each  $z' \in a^{(j)} + \Delta(2\epsilon)^m$
it follows from 5.26, 5.27 and 5.24 that the contour integral
$$(**) \hspace{10mm} \frac{1}{2\pi \text{i}} \int_{C_{r_{h,j}}} \frac{\partial \tilde{G}_t }{\partial z_{m+1}}(z' , z_{m+1}) \cdot (\tilde{G}_t (z' , z_{m+1}))^{-1} \hspace{1mm} dz_{m+1}$$
is well defined. It counts the number of zeros (with multiplicity) of the function $\tilde{G}_t (z' , \cdot)$ lying within
the circle $C_{r_{h,j}}$. Further, by 5.19, 5.24 and 5.27,  the integral is a continuous function of $\langle t, z' \rangle$ in the stated domain,
and so is constant there. Let its value be $q_{h,j}$
and let $Z(t,z') = \langle \rho_{1} (t,z'),\ldots,\rho_{q_{h,j}}(t,z') \rangle$ be a listing of the zeros of
$\tilde{G}_t (z' , \cdot)$ lying within the circle $C_{r_{h,j}}$ (each one counted according to its multiplicity).

Now, for $t \in T^{\dag}_0 \cap  (t^{(h)} + \Delta(2\epsilon)^N)$, $\langle z_2,\ldots,z_{m+1} \rangle \in a^{(j)} + \Delta(2\epsilon)^m$, \newline
$u \in [-1,1]^k \setminus E$,
and each $l=1,\ldots,q_{h,j}$, we  have, by 5.26 and the fact that $r_{h, j} < R_1$, that
$$z_2,\ldots,z_{m+1}, \rho_{l} (t,z_2,\ldots,z_{m+1}) \in \Delta(R_1)$$
and hence that the function
$$L_{t,u}^{h,j} : \C \times (a^{(j)} + \Delta(2\epsilon)^m) \to \C$$
given by

\vspace{2mm}

\noindent
5.28 \hspace{2mm} $L_{t,u}^{h,j}( z_1,z_2,\ldots,z_{m+1}) := \prod_{l=1}^{q_{h,j}} \tilde{H}_u (z_1, z_2,\ldots,z_{m+1}, \rho_l (t,z_2,\ldots,z_{m+1}))$

\vspace{2mm}

\noindent
is a monic polynomial of degree at most $d\cdot q_{h,j}$ in $z_1$  (by 5.23). Note that it is well-defined
since $\tilde{H}_u$ has domain $\C \times \Delta(R')^{m+1}$ and $R_{1} < R'$. (It is certainly possible
that $q_{k,j} = 0$, in which case we intepret the empty product as $1$ and the unique monic polynomial of degree $0$
as the constant function $1$.)

\vspace{2mm}

Now, since $L_{t,u}^{h,j}( z_1,z_2,\ldots,z_{m+1})$ is symmetric in the  $\rho_l (t,z_2,\ldots,z_{m+1})$
(i.e. it does  not depend on our particular ordering of the list $Z(t, z_2,\ldots,z_{m+1})$),
it follows easily that it is a definable function of all the variables $t,u,z_1,\ldots,z_{m+1}$ (restricted to the stated domain) and,
as a standard argument shows, it is holomorphic in $z_1,z_2,\ldots,z_{m+1}$. (Here one uses the generalization of (**) giving the integral
representation of sums of powers of functions of the roots of $\tilde{G}_t (z', \cdot )$. Arbitrary symmetric polynomial combinations of such functions
are then given as polynomials in these power sums.)

We now scale and translate the function $L_{t,u}^{h,j}$ by setting

\vspace{2mm}

\noindent
5.29 \hspace{2mm} $P_{t,u}^{h,j}( z_1,z_2,\ldots,z_{m+1}) := L_{t,u}^{h,j}( z_1,a_1^{(j)} + \epsilon z_2,\ldots, a_{m}^{(j)} + \epsilon z_{m+1})$

\vspace{2mm}

\noindent
so that each $P_{t,u}^{h,j}$ maps $\C \times \Delta(2)^m$ to $\C$, is a monic polynomial of degree at most $dq_{h,j}$
in $z_1$, and is bounded by $K^{q_{h,j}}$. (Notice that this holds true, by our convention concerning the monic polynomial
of degree $0$, even if $q_{h,j} = 0$.)

We now combine the functions $P_{t,u}^{h,j}$  as $h$ and $j$ vary over $\{ 1,\ldots,M \}$. Firstly, for $h \in \{ 1,\ldots,M \}$,
$u \in [-1,1]^k \setminus E$ and $t  \in T^{\dag}_0 \cap  (t^{(h)} + \Delta(2\epsilon)^N)$ define
\vspace{2mm}

\noindent
5.30 \hspace{2mm} $P_{t,u}^{h} := \prod_{j=1}^{M} P_{t,u}^{h,j}$

\vspace{2mm}

\noindent
so that  each $P_{t,u}^{h}$ maps $\C \times \Delta(2)^m$ to $\C$, is a monic polynomial of degree at most $d_h := \sum_{j=1}^{M} dq_{h,j}$
in $z_1$, and is bounded by  $(K+1)^{Mq_{h}}$,
where $q_h := \text{max} \{ q_{h,j} : j = 1,\ldots,M \}$.

Finally, we set

\vspace{2mm}

\noindent
5.31 \hspace{2mm} $\Lambda^{*} := \bigcup_{h=1}^{M} \{ P^h_{t,u} : u \in [-1,1]^k , t \in T_0^{\dag} \cap (t^{(h)} + \Delta(\epsilon )^N ) \}$.

\vspace{2mm}

Then $\Lambda^{*}$ is a $(2, m+1, (K+1)^{Mq})$-family, where $q:= \text{max} \{ q_h : h= 1,\ldots,M \}$,
each element of which is a monic polynomial of degree at most $\text{max} \{ d_h : h=1,\ldots,M \}$
in its first variable.

\vspace{2mm}

We now verify (*) (stated just before 5.11) which will complete the proof. We have
to show that if $u \in [-1,1]^k \setminus E$, then there exists $F \in \Lambda^{*}$ such that
$X_u \subseteq \{ x  \in [-1, 1]^n : \exists w \in [-1, 1]^m \bigwedge_{i=1}^n F(x_i ,w ) =0 \}.$

So let such a $u$ be given. Choose $t^{\dag} \in T_0^{\dag}$ as in 5.22. By 5.25 we may choose
$h \in \{ 1,\ldots, M \}$ such that $t^{\dag} \in t^{(h)} + \Delta(\epsilon )^N$. We let our $F$ be the function
$P_{t^{\dag},u}^h$ (see 5.30), which of course lies in $\Lambda^{*}$ (see 5.31). Now pick any $x = \langle x_1,\ldots,x_n \rangle \in X_u$.
By 5.22 we may pick $v = \langle v' , v_{m+1} \rangle \in (-R'',R'')^{m+1}$ such that  $\bigwedge_{i=1}^n \tilde{H}_{u} (x_i , v)=0 \wedge \tilde{G}_{t^{\dag}}(v)=0$.
By 5.26, there exists $j \in \{ 1,\ldots,M \}$ such that $v' \in a^{(j)} + (-\epsilon,\epsilon)^m$. Now, since $v_{m+1}$ lies within the circle $C_{r_{h,j}}$
and is a zero of the function
$\tilde{G}_{t^{\dag}}(v', \cdot)=0$,
it follows that $q_{h,j} > 0$ and that for some $l=1,\ldots, q_{i,j}$, we have $v_{m+1} = \rho_l (t^{\dag}, v')$.
Thus $\bigwedge_{i=1}^n \tilde{H}_{u} (x_i , v', \rho_l (t^{\dag}, v' ))=0$, and hence $\bigwedge_{i=1}^n L_{t^{\dag},u}^{h,j} (x_i , v') =0$ (see 5.28).
We now choose $w \in [-1,1]^m$ such that $v'= a^{(j)} + \epsilon w$. Then $\bigwedge_{i=1}^n P_{t^{\dag},u}^{h,j} (x_i , w) =0$ (see 5.29). It follows
that $\bigwedge_{i=1}^n P_{t^{\dag},u}^{h} (x_i , w) =0$ (see 5.30), i.e. $\bigwedge_{i=1}^n F(x_i , w) =0$, and we are done.

\bibliographystyle{amsplain}
\bibliography{anbib}

\end{document}